\documentclass[final]{siamltex}
\usepackage[hmargin=2.5cm,vmargin=3cm]{geometry}
\usepackage{graphicx}
\usepackage{amssymb}
\usepackage{amsmath}
\usepackage{epstopdf}
\usepackage{algorithm}
\usepackage{algorithmic}
\usepackage{setspace}
\usepackage{url}
\usepackage{color}
\usepackage{cancel}

\usepackage{lscape}

\usepackage{bbm}

\newcommand{\indicator}[1]{\mathbbm{1}_{\left[ {#1} \right] }}

\newcommand{\eps}{\varepsilon}

\newtheorem{remark}{ Remark}

\newtheorem{example}{Example}[section]
\newtheorem{condition}{Condition}[section]
\DeclareMathOperator*{\argmax}{arg\,max}

\begin{document}
\baselineskip 0.58cm
\bibliographystyle{apalike}

\title{Filtering the maximum likelihood for multiscale problems}
\author{Andrew Papanicolaou\footnotemark[2]\and Konstantinos Spiliopoulos\footnotemark[3]}
\footnotetext[2]{Department of ORFE, Princeton University, Sherrerd Hall, Charlton Street, Princeton NJ 08544, \em{apapanic@princeton.edu}. Work partially supported by NSF grant DMS-0739195.}
\footnotetext[3]{Department of Mathematics \& Statistics, Boston University, 111 Cummington Street, Boston MA 02215, \em{kspiliop@math.bu.edu}.
Work partially supported  by NSF grant DMS-1312124.}

\maketitle

\begin{abstract}
Filtering and parameter estimation under partial information for multiscale diffusion problems is studied in this paper. The nonlinear filter converges in the mean-square
sense to a filter of reduced dimension. Based on this result, we  establish that the conditional (on the observations) log-likelihood process has a correction term given
by a type of central limit theorem. We prove that an appropriate normalization of the log-likelihood minus a log-likelihood of reduced dimension converges weakly to a
normal distribution. In order to achieve this we assume that the operator of the (hidden) fast process has a discrete spectrum and an orthonormal basis of eigenfunctions. We then propose to estimate the unknown model parameters using the reduced log-likelihood, which is beneficial because reduced dimension means that there is significantly less runtime for this optimization program. We also establish consistency and asymptotic normality of the maximum likelihood estimator. Simulation results illustrate our theoretical findings.
\end{abstract}
\begin{keywords}Ergodic filtering, fast mean reversion, homogenization, Zakai equation, maximum likelihood estimation, central limit theory.\\
{\bf Subject classifications. } 93E10 93E11 93C70
\end{keywords}

\section{Introduction} \label{S:Introduction}
In this paper we consider the problem of filtering and parameter estimation for stochastic differential equations (SDEs) with multiple time scales. The model has parameter $0<\delta\ll 1$ that separates the slow and fast scales of the system, and it is assumed that $\delta$ is known a priori. The filtering problem involves two
 SDEs: a hidden ergodic diffusion process $X^{\delta}$ whose solution is known to be a path from an  SDE with a fast time scale of $1/\delta$, and an observation $Y^{\delta}$ that depends on $X^{\delta}$ but evolves in a slow time scale that is of order 1. The parameter estimation problem arises when the SDE satisfied by $(Y^{\delta},X^{\delta})$  has an unknown parameter $\theta\in\Theta$ where $\Theta\subset\mathbb{R}^{d}$.

Under the appropriate conditions, the nonlinear filter converges in a mean-square sense to a homogenized filter of reduced
 dimension. Based on this result and under the additional assumption that the infinitesimal generator of the fast process has a discrete spectrum with an orthonormal basis of eigenfunctions, we establish a central limit theorem (CLT) for the (conditional) log-likelihood. In particular, we prove that the difference of the log-likelihood (in other words, the log of the solution to the Zakai equation with input test function of $f\equiv 1$) minus a  log-likelihood of reduced dimension, normalized by $\sqrt{\delta}$, converges weakly to a centered normal distribution with a variance that is a function of the model parameters. To the best of the authors' knowledge, the CLT proven in this paper is the first of its kind. We also establish consistency and asymptotic normality of the maximum likelihood estimator (MLE) of the reduced log-likelihood. Compared to the original log-likelihood, the computation of the MLE based on the reduced log-likelihood is simpler and faster to compute.

This work is related to other works in filtering, wherein the observed process evolves in a slower scale than the hidden process. In \cite{Kushner}, it is shown that the difference of the unnormalized actual filter and its homogenized counterpart goes to zero in distribution for fixed test functions. The authors in \cite{BensoussanBlankenship1986,Ichihara2004} study homogenization of nonlinear filtering based on asymptotic analysis of a dual representation of the filtering equation. The authors in \cite{ParkSriSowers2008, ParkRozovskySowers2010,ParkSriSowers2011,ImkellerSriPerkowskiYeong2012} prove convergence in probability and in the $pth$-norm (in the latter article) of the nonlinear filter to its homogenized version. Notably, in \cite{ImkellerSriPerkowskiYeong2012} the authors use a formulation through backward SDEs and make use of the estimates for the related transition probability densities of \cite{PardouxVeretennikov2}; they also obtain rates of convergence in $L^{p!
 }$.  In \cite{KleptsinaLiptserSerebrovski1997}, the authors prove convergence of the filter in mean square sense and in a quite general setting; they assume convergence of the total variation norm of $(Y^{\delta},X^{\delta})$ and also assume convergence in probability of the slow part of the hidden component.

Parameter estimation problems for partially observed processes have been also studied elsewhere in the literature, e.g., \cite{Kutoyants, JamesGland1995}, although the effect
of multiple scales was not studied there. Moreover, in \cite{PapavasileiouPaviotisStuart2009} the authors study maximum likelihood estimation for fully-observed systems (not partially observed as in our case) of multiscale processes where the fast process takes values on a compact set.

The aforementioned existing literature has focused on proving convergence of the nonlinear filter to a filter of reduced dimension, namely to understand the dominant limiting behavior.
In this paper, we are interested in parameter estimation for such models. Thus, for statistical inference purposes we need to prove that the filter will be close to a filter of reduced dimension for any parameter value (and not just for the true parameter value), with closeness referring to either convergence in probability or mean square under the measure parameterized by the true parameter value. We establish that this result is true in the $L^{2}$-sense and also show that convergence results in the existing literature can be extended to a class of unbounded test functions that have more than two moments. Then, we obtain a CLT for the difference between the log-likelihood function and the log-likelihood from the filter of reduced dimension. To obtain the CLT, we further assume that the infinitesimal generator of the fast process has a discrete spectrum with an orthonormal basis of eigenfunctions. The difference in the log-likelihood functions is of order $\sqrt{\delta}$, a!
 nd we are able to state explicitly the variance of the limiting centered normal distribution. We emphasize that the filter of reduced dimension uses the original observations, which are the only available observations, and hence, the results justify using the reduced log-likelihood for purposes of statistical inference. For computational purposes, it is simpler and much faster to implement the filter of reduced dimension than it is for the original log-likelihood.

Filtering is a well established area and some general references for stochastic nonlinear filtering are \cite{bainCrisan,Kallianpur,Kushner,Rozovskii}. Our motivation for studying parameter estimation for partially observed multiscale diffusion models comes from financial applications, e.g., convenience yield in commodities markets or estimation of latent states in markets with high frequency trading (HFT). For example, non-predatory HFTs lead to increased liquidity and faster price discovery. Hence, a change-point detection algorithm on HFT data can be used to determine when price discovery has occurred. Another application could be the detection of an increased bid-ask spread which may correspond to increased volatility. We refer the reader to \cite{hendershot,zhang} for related discussions.

The rest of the paper is organized as follows: Section \ref{S:model} presents the system of equations that we consider, states our main assumptions, and restates fundamental results from filtering theory. Section \ref{S:filterLimit} presents our results on the asymptotic properties of the filter and of the log-likelihood. In particular, in Subsection \ref{SS:filterLimit} we discuss the $L^{2}$-convergence of the nonlinear filter, a result which is used in Subsection \ref{SS:PreparatoryCLT} to establish the CLT for the log-likelihood; the CLT is the main result. These results are then used in Section \ref{S:ParameterEstimation} to justify the claim that parameter estimation can be based on the reduced system, where we prove consistency and asymptotic normality of the MLE of the reduced log-likelihood. A simulation study illustrating the theoretical results is  presented in Section \ref{SS:simExample}. Conclusions are in Section \ref{S:Conclusions}. For presentation purposes, !
 most of the proofs are deferred to Appendices \ref{A:FilterConvergence1} and \ref{A:lemmasAndProps}.
\section{Formulation of Problem and Known Preliminary Results} \label{S:model}
On a probability space $(\Omega,(\mathcal F_t)_{t\leq T},\mathbb P)$ with $T<\infty$, for positive integers $m,n$ we consider the $(m+n)$-dimensional process $(X^{\delta},Y^{\delta})=\{(X^{\delta}_{t},Y^{\delta}_{t})\in\mathbb R^m\times\mathbb R^n, 0\leq t\leq T\}\in \mathcal C([0,T];\mathbb R^m\times\mathbb R^n)$, which satisfies a system of stochastic differential equations (SDE's)
\begin{eqnarray}
dY^{\delta}_{t}&=&h_{\theta}\left(X^{\delta}_{t}\right)dt + dW_{t}\nonumber\hspace{4cm}\hbox{(observed)}\\
dX^{\delta}_{t}&=& \frac{1}{\delta}b_{\theta}\left(X^{\delta}_{t}\right)dt + \frac{1}{\sqrt{\delta}}\sigma_{\theta}\left(X^{\delta}_{t}\right)dB_{t}\label{Eq:Model3}\hspace{2cm}\hbox{(hidden)}
\end{eqnarray}
where $(W_t)_{t\leq T}$ and $(B_t)_{t\leq T}$ are (unobserved) independent Wiener processes in $\mathbb{R}^{n}$ and $\mathbb R^m$, respectively. Our general assumptions on the functions $h_\theta,b_\theta$ and $\sigma_\theta$ are given in Section \ref{SS:coefficients}, but some of our theorems will require a stronger assumption on the spectrum of the infinitesimal generator of the $X$-process given in Section \ref{SS:SpectralDecomposition}. We assume that the parameter $\theta$ is also unknown, but takes values in a set $\Theta\subset\mathbb{R}^{d}$ with $d$ being a positive integer. Initially, the process $X_0^\delta$ is distributed according to a given prior distribution, and from here forward we  take $Y_0=0$. We denote the probability measure with $\mathbb P$, but we work with the parameterized family $(\mathbb P_\theta)_{\theta\in\Theta}$ in order to denote probabilities that are conditional on the parameter value,

\[\mathbb P_\theta((X^\delta,Y^\delta)\in\mathcal B) = \mathbb P\left((X^\delta,Y^\delta)\in\mathcal B\Big|~\hbox{$\theta$ is parameter in equation \eqref{Eq:Model3}}\right)\qquad\forall\theta\in\Theta ,\]
for any Borel set $\mathcal B\subset C([0,T],\mathbb R^m\times\mathbb R^n)$, and we let $\mathbb E_\theta$ denote its expectation operator. The parameter value to be estimated is the true (but unknown) value of $\theta$; we denote the true value by $\alpha\in\Theta$.

Our goal for this paper is to develop a theoretical framework allowing statistical inference on the unknown parameter  $\theta$ given an observed path $(Y_s^\delta)_{s\leq t}$ and assuming that $0<\delta\ll 1$. In particular, our goal in this paper is twofold:
\begin{enumerate}
\item{Obtain the limiting behavior and a central limit theorem (CLT) type correction for the posterior (on the observed path $(Y_s^\delta)_{s\leq t}$) likelihood function as $\delta\downarrow 0$.}
\item{Using the asymptotic behavior of the likelihood function, develop a framework for statistical inference for the unknown parameter $\theta$ given an observed path $(Y_s^\delta)_{s\leq t}$, assuming that $0<\delta\ll 1$.}
\end{enumerate}

In Subsection \ref{SS:coefficients} we establish notation and conditions guaranteeing ergodicity and that the filtering problem is well posed. Then, in Subsection \ref{SS:SpectralDecomposition}, we introduce a more specific framework wherein the infinitesimal generator of the fast process $X^{\delta}$ with $\delta=1$ has a discrete spectrum with an orthonormal basis of eigenfunctions, which allows us to establish the CLT of Theorem \ref{thm:CLT}. Then, in Subsection \ref{S:FilteringReview} we review some known, useful results from filtering theory.

\subsection{Notation and General Assumptions}\label{SS:coefficients}
Let $a,b$ be two vectors in some Euclidean space, say $\mathbb{R}^{n}$. For notational convenience we shall often write $a\cdot b$ or simply $ab$ for their inner product and we will denote by $|\cdot|$ the standard Euclidean norm.

Moreover, we denote by $\mathcal{X}=\mathbb{R}^{m}$ the state space of the fast component $X$. For any $f\in\mathcal{C}^{2}(\mathcal{X})$, we define the set of operators $\left(\mathcal L_\theta\right)_{\theta\in\Theta}$ such that
\begin{equation}
\mathcal{L}_\theta f(x)=b_\theta(x)\cdot D_{x}f(x)+\frac{1}{2}\textrm{tr}\left[\sigma_\theta(x)\sigma_\theta^{T}(x)D^{2}_{x}f(x) \right]\ ,
\end{equation}
where $D_x$ is the gradient operator. From \eqref{Eq:Model3} it follows that $\frac 1\delta\mathcal L_\theta$ is the infinitesimal generator of $X_t^\delta$.

We will make several assumptions on the growth and smoothness of the coefficients in order to guarantee that (\ref{Eq:Model3}) has a well-defined strong solution, that the fast component $X_t^\delta$ is ergodic, that the slow component $Y_t^\delta$ has a well defined homogenization limit as $\delta\downarrow 0$ in the appropriate sense, and that the filtering equations make sense. A set of assumptions that guarantee these properties are contained in the following condition (see \cite{PardouxVeretennikov2} for ergodic theory where they consider parts i) through iv) given below, and also Chapter 3 of \cite{bainCrisan} for filtering):
\begin{condition}\label{A:Assumption1}
\begin{enumerate}
\item{In order to guarantee the existence of an invariant measure $\mu_{\theta}(dx)$ for $X^1$ (i.e., for the process $X_{t}^{\delta}$ with $\delta=1$) we assume that
\[
\lim_{|x|\rightarrow\infty}\sup_{\theta\in\Theta} b_{\theta}(x)\cdot x=-\infty.
\]
}
\item{To guarantee uniqueness of the invariant measure for $X^1$, we assume that $\sigma_{\theta}(x)\sigma_{\theta}^{T}(x)$ is uniformly non-degenerate in $\theta$, i.e., there exist constants $c(\theta)>0$ such that for all $x\in\mathcal{X}$
\[
\left|\xi\sigma_{\theta}(x)\right|^{2}\geq c(\theta)|\xi|^{2},\quad\textrm{ for all $(\theta,\xi)\in\Theta\times\mathbb{R}^{n}$ and for all $x\in\mathbb R^n$}.
\]}
\item{$\sigma_{\theta}(x)\sigma_{\theta}^{T}(x)$ is bounded  in $(\theta,x)\in\Theta\times\mathcal{X}$ and $\sigma_{\theta}(x)$ is globally Lipschitz in $x\in\mathcal{X}$ uniformly in $\theta\in\Theta$.}
\item{$b_{\theta}(x)$ is locally bounded and globally Lipschitz in $x\in\mathcal{X}$, uniformly in $\theta\in\Theta$.}
\item{$h_{\theta}\in C(\mathcal{X})$, is locally bounded and globally Lipschitz in $x\in\mathcal{X}$, uniformly in $\theta\in\Theta$.}
\item{$X^{\delta}_{0}=X_{0}$ is a continuous random variable such that $\mathbb{E}|X_{0}|^{3}<\infty$.}
\item{The functions $h_{\theta},b_{\theta},\sigma_{\theta}$ are Lipschitz continuous in $\theta\in\Theta$ and $\Theta \subset\mathbb{R}^{d}$ is compact.}
\end{enumerate}
\end{condition}
\begin{remark}\label{R:RelaxedAssumptions}
A typical example of a process $X$ that satisfies Condition (\ref{A:Assumption1}) is the Ornstein-Uhlenbeck process of Example \ref{Ex:OU} that we present below. One can verify that our results also hold for certain degenerate processes, such as the square root process (CIR) of Example \ref{Ex:CIR} where $\sigma(x)=\sqrt{x}$, i.e., it  degenerates at $x=0$ but nevertheless it is ergodic; we do not analyze these special cases in this paper.
\end{remark}

For any function $f\in L^2(\mathcal{X},\mu_{\theta})$, denote its average with respect to the invariant measure $\mu_{\theta}(dx)$ as
\[
\bar f_\theta= \int_{\mathcal X}f(x)\mu_\theta(dx)\ .
\]
It is a well known result that $Y^{\delta}_{\cdot}$ converges in distribution in $C([0,T];\mathbb{R}^n)$ to the process $\overline{Y}_{\cdot}$
 (e.g. \cite{BLP,PardouxVeretennikov2}), where
\begin{equation}
\overline{Y}_{t}=\bar{h}_{\theta} t+ W_{t}.\label{Eq:LimitingModel}
\end{equation}
Actually, due to the fact that the observation process $Y^{\delta}_{t}$ has constant diffusion, Condition \ref{A:Assumption1} and the ergodic theorem guarantee that a stronger result holds for any $\theta\in\Theta$, i.e., for every $\varepsilon>0$
\begin{equation}
\mathbb{P}_\theta\left(\sup_{0\leq t\leq T}\left|Y^{\delta}_{t}-\overline{Y}_{t}\right|\geq \varepsilon \right)\rightarrow 0, \textrm{ as }\delta\downarrow 0\qquad\forall\theta\in\Theta.\label{Eq:MeanSquareConvergence}
\end{equation}

\subsection{Spectral Decomposition}\label{SS:SpectralDecomposition}
A stronger assumption than Condition \ref{A:Assumption1} is that the operator $\mathcal L_\theta$ has a discrete spectrum with an orthonormal basis of eigenfunctions. Some of the theorems in this paper do not require such strong assumptions on the operator's spectrum (e.g. Theorems \ref{T:FilterConvergence1}, \ref{T:ConsisitencyReducedLikelihood} and \ref{T:CLTReducedLikelihood} do not rely on discrete spectrum and orthonormal eigenfunctions), but the proof of the CLT in Theorem \ref{thm:CLT} relies on $\mathcal L_\theta$'s spectrum having these properties.

The steps taken in proving Theorem \ref{thm:CLT} utilize the spectral expansion of functions $f\in L^{2}(\mathcal{X},\mu_{\theta})$ with respect to the eigenfunctions of the operator $\mathcal{L}_{\theta}$. We say that the class of operators $\{\mathcal L_\theta\}_{\theta\in\Theta}$ has a discrete spectrum if for each $\theta\in\Theta$ there are eignenvalues $(-\lambda_i^\theta)_{i=0,1,2,3,\dots}$ such that
\[0=\lambda_0^\theta>-\lambda_1^\theta\geq-\lambda_2^\theta\geq\dots\ .\]
For each $i\geq 0$ we denote the $ith$ eigenfunction as $\psi_i^\theta(x)$ such that
\[\mathcal L_\theta\psi_i^\theta = -\lambda_i^\theta\psi_i^\theta\]
and we assume for each $\theta\in\Theta$ that the eigenfunctions form an orthonormal basis of $L^2(\mathcal{X},\mu_{\theta})$ so that
\[\int \psi_i^\theta(x)\psi_j^\theta(x)\mu_\theta(dx) = 1_{[i=j]} ,\]
and any square-integrable function $f\in L^2(\mathcal{X},\mu_{\theta})$ can be written as
$f(x) = \sum_{i=0}^\infty \psi_i^\theta(x)\left<f,\psi_i^\theta\right>_\theta$, where $\left<f,\psi_i^\theta\right>_\theta = \int f(x') \psi_i^\theta(x')\mu_\theta(dx')$. Notice that because $\mathcal L_\theta$ is a differential operator and the spectral elements are assumed to be an orthonormal basis, we get that $\psi_0^\theta\equiv 1$. This means
\begin{equation}
\left<\psi_i^\theta,1\right>_\theta=\left<\psi_i^\theta, \psi_0^\theta\right>_\theta = 0\qquad\hbox{for }i=1,2,3,\dots\ .\label{Eq:OrthonormalBasis}
\end{equation}
Below we consider some examples of processes whose operators have discrete spectrum with an orthonormal basis of eigenfunctions:

\begin{example}\label{Ex:OU} A non-degenerate ergodic process with a discrete spectrum is the 1-dimensional Ornstein-Uhlenbeck (OU) process,
\[dX_t = \kappa(\theta - X_t)dt+\sigma\sqrt{2} dB_t\]
where $\theta\in\Theta\subset\mathbb R$ and $\sigma,\kappa>0$. The eigenvalues of $\mathcal L_\theta$ are $0,-1,-2,-3,\dots$, and the Hermite polynomials form an orthonormal basis. Moreover, this process is ergodic with invariant measure Gaussian and in particular $\mu_{\theta}(dx)=\sqrt{\frac{\kappa}{2\pi \sigma^2 }}e^{-\frac{\kappa(x-\theta)^{2}}{2\sigma^2}}dx$.
\end{example}
\begin{example}\label{Ex:CIR} A degenerate ergodic process with a discrete spectrum is the 1-dimensional Cox-Ingersol-Ross (CIR) process,
\[dX_t = \kappa(\theta - X_t)dt+\sqrt{2\sigma^2 X_t}dB_t\]
where $\theta\in\Theta\subset\mathbb R^+$ and $\kappa>0$. The eigenvalues of $\mathcal L_\theta$ are $0,-1,-2,-3,\dots$, and the (generalized) Laguerre polynomials form an orthonormal basis. Moreover,  if $\kappa \theta>\sigma^2$ then this process is ergodic with invariant measure, the measure for a gamma distribution and in particular $\mu_{\theta}(dx)=\frac{a^{\beta}}{\Gamma(\beta)}x^{\beta-1}e^{-a x}dx$, where $\Gamma(\cdot)$ is the gamma function, $a=\kappa/\sigma^2$ and $\beta=\kappa\theta/\sigma^2$. Even though this SDE does not satisfy Condition \ref{A:Assumption1}(ii)-(iii), the SDE has a unique strong solution which is ergodic and thus one expects the results of this paper to hold.
\end{example}
%

We conclude with a multidimensional example.
\begin{example}\label{Ex:MultiDOU} A non-degenerate ergodic process with a discrete spectrum is the $m$-dimensional linear SDE,
\[dX_t = -A X_tdt+\Gamma dB_t\]
where $A$ is $m\times m$ positive definite and $\Gamma$ is a matrix of appropriate dimensions, such that $(A,\Gamma)$ is a controllable pair. This process is ergodic and its infinitesimal generator has discrete spectrum. The orthonormal basis can be constructed by taking products of the modified Hermite functions for each variable, see \cite{LiberzonBrockett,linetskyBook} for more details and analysis.
\end{example}

\subsection{Filtering Equations}\label{S:FilteringReview}
Our data is contained in the filtration generated by the observed path, which is the $\sigma$-algebra $\mathcal Y_t^\delta \doteq\mathcal F_t^{Y^\delta} = \sigma \{(Y_s^\delta)_{s\leq t}\}$. The filtration $\mathcal Y_t^\delta$ does not reveal  the true but unknown parameter value $\alpha\in\Theta$. However, we can compute a posterior distribution conditional on a given parameter value, and then perform further statistical inference such as maximum likelihood in order to estimate the true parameter value. For a general introduction to stochastic filtering we refer the reader to classical manuscripts, such as \cite{bainCrisan,Kallianpur,Kushner,Rozovskii}.

For any $\theta\in\Theta$ (and not just the true parameter value, $\alpha\in\Theta$, that has generated the data in $\mathcal Y_t^\delta$), let's define the exponential martingale $Z_T^{\delta,\theta}$ which gives a new measure $\mathbb{P}_{\theta}^{*}$ on $(\Omega,\mathcal{F})$, such that
\begin{equation}
\frac{d\mathbb{P}_{\theta}}{d\mathbb{P}^{*}_{\theta}}\doteq Z_T^{\delta,\theta}=\exp\left\{ \int_{0}^{T}h_{\theta}(X^{\delta}_{s})dY^{\delta}_{s}-\frac{1}{2}\int_{0}^{T}\left|h_{\theta}(X^{\delta}_{s})\right|^{2}ds   \right\}\ .\label{Eq:LogLikelihoodPreLimit}
\end{equation}
By Girsanov's theorem on the absolutely continuous change of measure in the space of trajectories in $\mathcal{C}([0,T],\mathbb R^m)$, the probability measures $\mathbb{P}_{\theta}$ and $\mathbb{P}^{*}_{\theta}$ are absolutely continuous with respect to each other, and the distribution of $X^\delta$ is the same under both $\mathbb{P}_{\theta}$ and $\mathbb{P}_{\theta}^{*}$. Furthermore,  the process $Y^\delta$ is a $\mathbb{P}_{\theta}^*$-Brownian motion independent of $X^\delta$, and $Z^{\delta,\theta}$ is a $\mathbb P_\theta^*$-martingale.

Next, for  $f:\mathcal X\rightarrow\mathbb R$ such that $\mathbb E_\theta^*|f(X_t^\delta)|^2<\infty$, we define the measure valued process $\phi^{\delta,\theta}_{t}$ acting on $f$ as
\begin{equation}\label{eq:expFormulation}
\phi^{\delta,\theta}_{t}[f] \doteq\mathbb{E}_{\theta}^{*}\left[ Z_t^{\delta,\theta}f(X^{\delta}_{t}) \Big |\mathcal{Y}_{t}^\delta\right]\ ,
\end{equation}
a process which, for $f\in C^{2}_{c}(\mathcal{X})$, is well known to be the unique solution (see \cite{rozovsky1992}) to the following equation:
\begin{eqnarray}
d\phi^{\delta,\theta}_{t}[f]&=&\frac{1}{\delta}\phi^{\delta,\theta}_{t}[\mathcal{L}_{\theta}f]  dt + \phi^{\delta,\theta}_{t}[h_{\theta} f] dY_t^\delta,\quad \mathbb{P}_{\theta}^{*} \textrm{-a.s.},  \quad \phi_0^{\theta}[f]=\mathbb E_{\theta}f(X_0^\delta)\label{Eq:Zakai}
 \end{eqnarray}
Equation \eqref{Eq:Zakai} is the Zakai equation for nonlinear filtering. In the literature, the term `filter' refers to a posterior measure on $X_t^\delta$ given $\mathcal Y_t^\delta$, and so $\phi_t^{\delta,\theta}$ is also a filter. Specifically, the process $\phi_t^{\delta,\theta}$ is an unnormalized probability measure with $\phi_t^{\delta,\theta}[1]$ being the likelihood function, and the maximizer of $\phi_t^{\delta,\theta}[1]$ is the maximum likelihood estimator (MLE). In other words, given the observation $(Y_s^\delta)_{s\leq t}$, the MLE is
\begin{equation}
\label{Eq:MLE1}
\theta_t^\delta\doteq \argmax_{\theta\in\Theta}\phi_t^{\delta,\theta}[1]\ .
\end{equation}
Furthermore, we can apply the Kalianpour-Striebel formula to obtain the normalized filter,
\begin{equation}
\label{Eq:KSformula}
\pi^{\delta,\theta}_{t}[f]\doteq\mathbb E_{\theta}\left[f(X_t^\delta)\Big|\mathcal Y_t^\delta\right]=\frac{\phi_t^{\delta,\theta}[f]}{\phi_t^{\delta,\theta}[1]}\quad \mathbb{P}_{\theta},\mathbb{P}_{\theta}^{*} \textrm{-a.s.}\ .
\end{equation}
An important case is $f(x)=x$ because $X_t^\delta$ is often tracked with the posterior mean, $\widehat X_t^{\delta,\theta}\doteq \mathbb E_{\theta}[X_t^\delta|\mathcal Y_t^\delta].$ The posterior mean can be given by the Kalman filter when $\sigma_{\theta}$ does not depend on $x$ and there is linearity in $x$ for both $h_{\theta}$ and $b_{\theta}$. Another important case is $f(x)=h_{\theta}(x)$ because of the \textit{innovations process},
\[\nu_t^{\delta,\theta}\doteq Y_t^\delta- \int_0^t\pi^{\delta,\theta}_{s}[h_\theta]ds\qquad\forall t\in[0,T]\ ,\]
(recall we assumed that $Y_0 = 0$). The process $\nu_t^{\delta,\theta}$ is a $\mathbb P_{\theta}$-Brownian motion under the filtration $\mathcal Y_t^\delta$, but will only be observable as Brownian motion if $\theta=\alpha$, i.e. when the true parameter value is taken. For suitable test functions $f:\mathcal X\rightarrow \mathbb R$, the innovation is used in the nonlinear Kushner-Stratonovich equation to describe the evolution of $\pi_t^{\delta,\theta}[f]$,
\begin{equation}
\label{Eq:Kushner}
d\pi_t^{\delta,\theta}[f] = \frac 1\delta \pi^{\delta,\theta}_{t}[\mathcal L_{\theta} f] dt+ \left(\pi^{\delta,\theta}_{t}[f h_{\theta}]-\pi^{\delta,\theta}_{t}[f]\pi^{\delta,\theta}_{t}[h_{\theta}]\right)d\nu_t^{\delta,\theta}\quad \mathbb{P}_{\theta}\textrm{-a.s.}\ . \end{equation}
The innovations Brownian motion will be used in later sections where we consider asymptotics of the log-likelihood function.

\section{Asymptotic Results of the Filter and of the Likelihood Function}\label{S:filterLimit}

In this section we establish some results on the filter's convergence. In Subsection \ref{SS:filterLimit} we use the convergence results found in \cite{ImkellerSriPerkowskiYeong2012} (see also \cite{ParkSriSowers2008, ParkRozovskySowers2010,ParkSriSowers2011,ImkellerSriPerkowskiYeong2012}) to prove convergence in probability of the filter for a class of unbounded test functions (e.g. for the eigenfunctions of the operator $\mathcal L_\theta$). Then, subsection \ref{SS:PreparatoryCLT} will use these results to derive a CLT for the log-likelihood function, which is the main result of the paper.

Consider the `averaged' exponentials
\begin{equation}
\bar{Z}_t^{\delta,\theta}\doteq \exp\left\{ \bar{h}_{\theta} Y_t^\delta-\frac{1}{2}\left|\bar{h}_{\theta}\right|^{2}t   \right\}\ ,\quad \bar{Z}_t^{\theta}\doteq \exp\left\{ \bar{h}_{\theta}\overline{Y}_t-\frac{1}{2}\left|\bar{h}_{\theta}\right|^{2}t   \right\}\ .\label{Eq:DefinitionsZ}
\end{equation}
In fact the solution to the Zakai equation of \eqref{Eq:Zakai} is close in mean square sense  to a limiting filter based on $\bar Z_T^{\delta,\theta}$. For $f\in\mathcal{C}^{2}_{c}(\mathcal{X})$, we define new posterior measures $\bar{\phi}^{\delta,\theta}_{t}[f]$ and $\bar{\phi}^{\theta}_{t}[f]$ which satisfy the stochastic evolution equations
\begin{eqnarray}
\label{Eq:avgZakai}
d\bar{\phi}^{\delta,\theta}_{t}[f]&=&\frac{1}{\delta}\bar{\phi}^{\delta,\theta}_{t}[\mathcal{L}_{\theta}f]dt+\bar{\phi}^{\delta,\theta}_t[f] \bar{h}_{\theta} dY^{\delta}_t, \quad \bar{\phi}^{\delta,\theta}_{0}[f]=\mathbb{E}_{\theta}\{f(X_{0}^{\delta})\}\\
\label{Eq:avgZakaiLimit}
d\bar{\phi}^{\theta}_{t}[f]&=& \bar{\phi}^{\theta}_t[f] \bar{h}_{\theta} d\overline Y_t\, \quad  \bar{\phi}^{\delta,\theta}_{0}[f]=\bar{f}_{\theta}.
 \end{eqnarray}
It is straightforward to verify with It\^o's lemma that the `average' Zakai equations  \eqref{Eq:avgZakai} and \eqref{Eq:avgZakaiLimit} have solutions
\begin{eqnarray}
\label{Eq:solAvgZakai}
\bar{\phi}^{\delta,\theta}_{t}[f]&=&\mathbb E_{\theta}^*\left[ f(X_{t}^{\delta}) \bar{Z}_t^{\delta,\theta}\Big|\mathcal Y_t^\delta\right] = \mathbb E_{\theta}[ f(X_{t}^{\delta}) ]\bar{Z}_t^{\delta,\theta}\ ,\\
\label{Eq:solAvgZakaiLimit}
\bar{\phi}^{\theta}_{t}[f]&=& \bar{f}_{\theta} \bar{Z}_t^{\theta}\
\end{eqnarray}
We also define $\bar{\pi}^{\delta,\theta}_{t}[f]=\frac{\bar{\phi}^{\delta,\theta}_{t}[f]}{\bar{\phi}^{\delta,\theta}_{t}[1]}=\mathbb E_{\theta} f(X_{t}^{\delta}) $ and $\bar{\pi}^{\theta}_{t}[f]=\frac{\bar{\phi}^{\theta}_{t}[f]}{\bar{\phi}^{\theta}_{t}[1]}=\bar{f}_{\theta}$.

\begin{remark}
The results of this section (namely Theorems \ref{T:FilterConvergence1} and \ref{thm:CLT} and Corollaries \ref{Lemma:likelihoodConvergence} and \ref{C:CLT}) will justify the approximation of $\phi^{\delta,\theta}[1]$ by $\bar{\phi}^{\delta,\theta}[1]$ for statistical inference purposes. Notice that $\bar{\phi}^{\delta,\theta}[1]$ is associated with the actual data, i.e., it is associated with $Y_t^\delta$ and not with $\overline{Y}_t$. $\overline{Y}_t$ is only used as a vehicle to obtain the necessary convergence results. Issues related with statistical inference are explored in Section \ref{S:ParameterEstimation}.
\end{remark}

\subsection{Convergence of the Filter and of the Likelihood Function}\label{SS:filterLimit}
At this point we need to impose an additional assumption on $Z^{\delta,\theta}_{t}$. In particular, we assume
\begin{condition}
\label{A:Assumption2}
For any $\theta\in\Theta$, there is a $q\in(1,\infty)$ such that
\[
\sup_{t\in[0,T]}\sup_{\delta\in(0,1)}\mathbb{E}_{\theta}^{*}|Z^{\delta,\theta}_{t}|^{q}+\sup_{t\in[0,T]}\sup_{\delta\in(0,1)}\mathbb{E}_{\theta}|Z^{\delta,\theta}_{t}|^{-q}<\infty.
\]
\end{condition}
Let us consider the $q\in(1,\infty)$ from Condition \ref{A:Assumption2} and let $p\in(1,\infty)$ be such that $1/q+1/p=1$.  Now let $\eta>2(p^{2}-1)$ and define the following class of test functions
\begin{equation}
\label{Eq:A}
\mathcal{A}_{\eta}^{\theta}\doteq\left\{f\in C^{4}(\mathcal X)\cap L^{2}(\mathcal{X},\mu_{\theta}): \sup_{t\in[0,T]}\sup_{\delta\in(0,1)}\mathbb{E}_{\theta}\left|f(X^{\delta}_{t})\right|^{2+\eta}<\infty, ~\right\}.
\end{equation}
Before stating the convergence results, we make some remarks related to Condition \ref{A:Assumption2} and the set $\mathcal{A}_{\eta}^{\theta}$.
\begin{remark}
Notice that because $\delta$ is a time scale, we could have written the definition in \eqref{Eq:A} with only a supremum over $t\geq 0$, and it would be an equivalent definition. That is, $X_t^\delta$ equals in distribution to $X_{t/\delta}^1$, so $\sup_{t\in[0,T]}\sup_{\delta\in(0,1)}\mathbb{E}_{\theta}\left|f(X^{\delta}_{t})\right|^{2+\eta}=\sup_{t\in[0,T]}\sup_{\delta\in(0,1)}\mathbb{E}_{\theta}\left|f(X^1_{t/\delta})\right|^{2+\eta}=\sup_{t\geq 0}\mathbb{E}_{\theta}\left|f(X^1_{t})\right|^{2+\eta}$.
\end{remark}
\begin{remark}
\label{R:novikovBoundedH}
Condition \ref{A:Assumption2} holds automatically for any finite $q>1$ if $ h_{\theta}(x)$ is bounded, e.g., Lemma 6.7 in \cite{ImkellerSriPerkowskiYeong2012}. Moreover, any $f\in\mathcal{C}_{b}^{4}(\mathcal{X})$ will also satisfy $f\in \mathcal{A}_{\eta}^{\theta}$ for any $\eta\geq 0$.
\end{remark}
\begin{remark}
\label{R:eigenfunctionsInA}
Suppose $X_0$ is distributed according to its invariant distribution. Then $\mathcal A_\eta^\theta$ consists of all functions $f\in C^4(\mathcal X)$ such that $\int |f(x)|^{2+\eta}\mu_\theta(x)dx<\infty$. However, the orthonormal basis of eigenfunctions $(\psi_i^\theta)_{i=0}^\infty$ associated with the operator $\mathcal L_\theta$ (as described in Section \ref{SS:SpectralDecomposition}) are not generally contained in $\mathcal A_\eta^\theta$ if $\eta>0$, but the examples given earlier qualify. 
Examples \ref{Ex:OU}, \ref{Ex:CIR}, and \ref{Ex:MultiDOU} also have $\psi_i^\theta\in\mathcal A_\eta^\theta$ for $\eta>0$, because  $\psi_i^\theta$ are polynomials with moments of all order, and so there are certainly $2+\eta$ moments of $\psi_i^\theta(X_t^\delta)$.
\end{remark}

The first result of this section holds without the assumption of spectral expansions, and is stated in the following theorem:
\begin{theorem}\label{T:FilterConvergence1}
Assume Conditions \ref{A:Assumption1} and \ref{A:Assumption2}. For any $\alpha,\theta\in\Theta$, we have that, uniformly in $t\in[0,T]$, the following are true:
\begin{enumerate}
\item{Let $f\in C^{4}_{b}(\mathcal{X})$. Then,  for every $\varepsilon>0$
\[
\lim_{\delta\downarrow 0}\mathbb{P}_{\alpha}\left( \left|\phi_t^{\delta,\theta}[f]-\bar\phi_t^{\delta,\theta}[f]\right| \geq \varepsilon \right)=0
\]
}
\item{Assume that there is $\eta>0$ such that $f\in\mathcal{A}_{\eta}^{\theta}$. Then, we have convergence of the filters in mean square
\[\lim_{\delta\downarrow 0}\mathbb E_{\alpha}\left|\pi_t^{\delta,\theta}[f]-\bar \pi_t^{\delta, \theta}[f]\right|^2=0\ . \]
and, moreover,
\[\lim_{\delta\downarrow 0}\left|\bar \pi_t^{\delta,\theta}[f]-\bar \pi_t^{\theta}[f]\right|=0\hspace{.5cm}\hbox{in $\mathbb P_\alpha$ probability}\ .\]}
\end{enumerate}
\end{theorem}
\begin{proof} The proof of this theorem is in Appendix \ref{A:FilterConvergence1}. \end{proof}

In statistical inference, a useful corollary of Theorem \ref{T:FilterConvergence1} is the convergence of likelihood functions:
\begin{corollary}\label{Lemma:likelihoodConvergence}Assume  Conditions \ref{A:Assumption1} and \ref{A:Assumption2}. For any $\alpha,\theta\in\Theta$ and each $t\geq0$, we have
\[\phi_t^{\delta,\theta}[1]- \bar\phi_t^{\delta,\theta}[1]\rightarrow 0 \qquad\hbox{in  } \mathbb{P}_{\alpha}\hbox{-probability as }\delta\rightarrow 0\ .\]
\end{corollary}

We note that results similar to Theorem \ref{T:FilterConvergence1} appear elsewhere in the literature, e.g., \cite{KleptsinaLiptserSerebrovski1997,Ichihara2004, ParkSriSowers2008, ParkRozovskySowers2010,ParkSriSowers2011,ImkellerSriPerkowskiYeong2012}, but with slightly different assumptions and set up. The main difference is that Theorem \ref{T:FilterConvergence1}, when compared to the previous works, states the convergence result under the measure parameterized by the true parameter value (i.e. the measure under which the observations are made, where $\theta=\alpha$) with the filters converging for \textit{any} parameter value. In other words, we will `observe' the
filters converging to the reduced filter. Moreover, the convergence of the filters in Theorem \ref{T:FilterConvergence1} is for test functions that belong to the space $\mathcal{A}_{\eta}^{\theta}$, which can include unbounded functions such as the eigenfunctions of the OU processes in Example \ref{Ex:OU} and \ref{Ex:MultiDOU} (see Remark \ref{R:eigenfunctionsInA}). By assuming that $\psi_i^\theta\in\mathcal A_\eta^\theta$ for some $\eta>0$, we are able to prove the results in Subsection \ref{SS:PreparatoryCLT}.


\subsection{Asymptotic Normality of Likelihood Function}\label{SS:PreparatoryCLT}
We  proceed to the statement and proof of the CLT for the log-likelihood function. In particular, we find that the difference in the original  log-likelihood
minus the log-likelihood of reduced dimension, divided by  $\sqrt{\delta}$, yields a quantity that is asymptotically normal. In proving the CLT, we  make extensive use of the discrete spectrum and eigenfunction basis. In this section we shall also assume the following:
\begin{condition}\label{A:Assumption3}
For any $i,j\in\mathbb{N}$ and any $\theta\in\Theta$, we assume that
\begin{enumerate}
\item{There exists $C_h>0$ independent of $\theta$ such that $\|h_{\theta}\|_\infty \leq C_h$,}
\item $\mathcal L_{\theta}$ has discrete spectrum with orthonormal basis functions (as prescribed in Section \ref{SS:SpectralDecomposition}),
\item{There exists $\eta>0$ such that $\psi_{i}^{\theta}\in\mathcal{A}_{\eta}^\theta$, for all $\theta\in\Theta$ and $i\in\mathbb{N}$,}
\item{$\pi^{\theta}_{0}[\psi_{i}^{\theta}]<\infty$ for all $\theta\in\Theta$ and $i\in\mathbb{N}$.}
\end{enumerate}
\end{condition}
It is worth noting that Condition \ref{A:Assumption3} subsumes Condition \ref{A:Assumption2} because it places a bound on $h_\theta$ (see Remark \ref{R:novikovBoundedH}). Moreover, the assumption that $\mathcal L_{\theta}$ has discrete spectrum with orthonormal basis functions is useful because the Zakai equation for the eigenfunctions $\psi_i^{\theta}$ simplifies to
\begin{equation}
\label{eq:spectralZakai}
d\phi_t^{\delta,\theta}[\psi_i^{\theta}] =-\frac{\lambda_i^{\theta}}{\delta} \phi_t^{\delta,\theta}[\psi_i^{\theta}]dt+\phi_t^{\delta,\theta}[h_{\theta} \psi_i^{\theta}]dY^{\delta}_t .
\end{equation}
Applying It\^o's lemma to $\frac{\phi_t^{\delta,\theta}[\psi_i^{\theta}]}{\phi_t^{\delta,\theta}[1]}$ we have the Kushner-Stratonovich equation
\begin{equation}
\label{eq:KS}
d\left(\frac{\phi_t^{\delta,\theta}[\psi_i^{\theta}]}{\phi_t^{\delta,\theta}[1]}\right)=-\frac{\lambda_i^{\theta}}{\delta}\frac{\phi_t^{\delta,\theta}[\psi_i^{\theta}]}{\phi_t^{\delta,\theta}[1]}dt+\underbrace{\left(\frac{\phi_t^{\delta,\theta}[h_{\theta}\psi_i^{\theta}]}{\phi_t^{\delta,\theta}[1]}-\frac{\phi_t^{\delta,\theta}[h_{\theta}]\phi_t^{\delta,\theta}[\psi_i^{\theta}]}{\left(\phi_t^{\delta,\theta}[1]\right)^2} \right) }_{=cov^{\delta,\theta}\left(h_{\theta}(X_t^\delta),\psi_i^{\theta}(X_t^\delta)|\mathcal{Y}^{\delta}_{t}\right)}d\nu_t^{\delta,\theta}
\end{equation}
where $d\nu_t^{\delta,\theta} = dY^{\delta}_t - \frac{\phi_t^{\delta,\theta}[h_{\theta}]}{\phi_t^{\delta,\theta}[1]}dt=dY^{\delta}_t -\mathbb E^{\delta,\theta}[h_{\theta}(X^{\delta}_t)|\mathcal{Y}^{\delta}_{t}]dt$ is the innovations Brownian motion under $\mathbb{P}_{\theta}$. By Duhamel's  principle the solution is
\begin{equation}
\label{eq:KSsoln}
\frac{\phi_t^{\delta,\theta}[\psi_i^{\theta}]}{\phi_t^{\delta,\theta}[1]} = e^{-\frac{\lambda_i^{\theta} t}{\delta}}\phi_0^{\theta}[\psi_i^{\theta}]+\int_0^t e^{-\frac{\lambda_i^{\theta}(t-s)}{\delta}}cov^{\delta,\theta}\left(h_{\theta}(X^{\delta}_s),\psi_i^{\theta}(X^{\delta}_s)\Big|\mathcal{Y}^{\delta}_{s}\right)d\nu_s^{\delta,\theta}.
\end{equation}
Equivalently, we can write
\begin{equation}
\label{eq:KSsoln2}
\pi_t^{\delta,\theta}[\psi_i^{\theta}] = e^{-\frac{\lambda_i^{\theta} t}{\delta}}\pi_0^{\theta}[\psi_i^{\theta}]+\int_0^t e^{-\frac{\lambda_i^{\theta}(t-s)}{\delta}}
\left(\pi_s^{\delta,\theta}[h_{\theta}\psi_i^{\theta}]-\pi_s^{\delta,\theta}[h_{\theta}]\pi_s^{\delta,\theta}[\psi_i^{\theta}]\right)
d\nu_s^{\delta,\theta}.
\end{equation}
Equations \eqref{eq:KSsoln} and \eqref{eq:KSsoln2} are the key identities used to prove the CLT. However, there are some ergodic properties of the filter that are required to do the proof. Appendix \ref{A:lemmasAndProps} has these results; Section \ref{SS:CLT} states and proves the CLT.

\subsubsection{Statement of CLT and Proof}\label{SS:CLT}
In this section, we quantify the estimation error which occurs if the reduced log-likelihood is used in place of the full version. In particular, we establish that the error in the log-likelihood function will be normally distributed with standard deviation of order $O(\sqrt\delta)$.

By Lemma 3.9 in \cite{bainCrisan} we have
\[\log\left( \phi_t^{\delta,\theta}[1]\right) = \int_0^t \pi_s^{\delta,\theta}[h_{\theta}]dY^{\delta}_s -\frac 12\int_0^t\left| \pi_s^{\delta,\theta}[h_{\theta}]\right|^2ds\ .\]
Let us write $\tilde{h}_{\theta}(x)=h_{\theta}(x)-\bar{h}_{\theta}$ and notice that $\left<\tilde{h}_{\theta},1\right>_{\theta}=0$. Then we write
\begin{align}
\frac{1}{\sqrt\delta}\left(\log\left( \phi_t^{\delta,\theta}[1]\right)-\log\left( \bar\phi_t^{\delta,\theta}[1]\right)\right)&= \frac{1}{\sqrt\delta}\left(\int_0^t\left( \pi_s^{\delta,\theta}[h_{\theta}]-\bar h_{\theta}\right)dY_s^\delta -\frac 12\left(\int_0^t\left| \pi_s^{\delta,\theta}[h_{\theta}]\right|^2ds-\int_0^t\left| \bar h_{\theta}\right|^2ds\right)\right)\nonumber\\
&=J^{\delta}_{1}+ J^{\delta}_{2}\ ,\nonumber
\end{align}
where we have defined $J_1^\delta$ and $J_2^\delta$ as
\begin{align}
J^{\delta}_{1}&=\frac{1}{\sqrt\delta}\int_0^t\left( \pi_s^{\delta,\theta}[h_{\theta}]-\bar h_{\theta}\right)dY_s^\delta=\frac{1}{\sqrt\delta}\int_0^t \pi_s^{\delta,\theta}[\tilde{h}_{\theta}]d\nu_s^{\delta,\theta}+
\frac{1}{\sqrt\delta}\int_0^t\left| \pi_s^{\delta,\theta}[\tilde h_{\theta}]\right|^2ds+\frac{\bar{h}_{\theta}}{\sqrt\delta}\int_0^t \pi_s^{\delta,\theta}[\tilde h_{\theta}]ds\nonumber
\end{align}
and
\begin{align}
J^{\delta}_{2}&=-\frac{1}{2\sqrt{\delta}}\left(\int_0^t\left| \pi_s^{\delta,\theta}[h_{\theta}]\right|^2ds-\int_0^t\left| \bar h_{\theta}\right|^2ds\right)=-\frac{1}{2\sqrt{\delta}}\int_0^t\left| \pi_s^{\delta,\theta}[\tilde{h}_{\theta}]\right|^2ds-\frac{\bar h_{\theta}}{\sqrt{\delta}}\int_0^t  \pi_s^{\delta,\theta}[\tilde{h}_{\theta}] ds\ .\nonumber
\end{align}
Hence, we obtain the representation
\begin{align}
\nonumber
&\frac{1}{\sqrt\delta}\left(\log\left( \phi_t^{\delta,\theta}[1]\right)-\log\left( \bar\phi_t^{\delta,\theta}[1]\right)\right)\\
\label{eq:logLderiv}
&=\underbrace{\int_0^t\frac{1}{\sqrt\delta}\left( \pi_s^{\delta,\theta}[\tilde h_{\theta}]\right)d\nu_s^{\delta,\alpha}}_{(*)}+\underbrace{\int_0^t\frac{1}{\sqrt\delta}\left( \pi_s^{\delta,\theta}[\tilde h_{\theta}]\right)\left(\pi_s^{\delta,\theta}[h_{\theta}]-\pi_s^{\delta,\alpha}[h_{\alpha}]\right)ds}_{(**)}+\underbrace{\frac{1}{2\sqrt\delta}\int_0^t\left| \pi_s^{\delta,\theta}[\tilde h_{\theta}]\right|^2ds}_{(\dagger)}\ ,
\end{align}
where $\nu_t^{\delta,\alpha}$ is a $\mathbb P_\alpha$ Brownian motion (i.e. it is Brownian motion under the true parameter), but not for $\mathbb P_\theta$ with $\theta\neq \alpha$. Now recall that by Condition \ref{A:Assumption3}, for every $i\in\mathbb{N}$ we have $\psi_i^\theta\in\mathcal{A}_{\eta}^{\theta}$. This implies that there exists finite constants that may depend on $i, T$ and $\theta$ such that
\begin{equation}
\sup_{\delta\in(0,1),\rho\in[0,T]}\mathbb E_\theta\left[|\psi_i^\theta(X_\rho^\delta)|^2\right]\leq C(\psi_{i},T,\theta)\ ,
\label{Eq:BoundedTerm01}
\end{equation}
from which we define another constant
\begin{equation}
C_{i,j}(T,\theta)\doteq\left(\frac{\left|\pi^{\theta}_{0}[\psi_{i}^{\theta}]\pi^{\theta}_{0}[\psi_{j}^{\theta}]\right|}{\lambda_i^{\theta}+\lambda_j^{\theta}}+
  \left(C(\psi_{i},T,\theta)+C(\psi_{j},T,\theta)\right)\left(\frac{1}{\lambda_i^{\theta}}+\frac{1}{\lambda_j^{\theta}}\right)\right)\ .\label{Eq:BoundedTerm0}
\end{equation}
If the infinite sum of these constants converges, then we can prove the following CLT for the log-likelihood function:

\begin{theorem}  \label{thm:CLT} \textbf{(Likelihood CLT).} Assume Conditions \ref{A:Assumption1} and \ref{A:Assumption3}. Moreover, assume that there exists constants $C(\psi_{i},T,\theta)$ that satisfy (\ref{Eq:BoundedTerm01}) such that for all $\theta\in\Theta$
\[
\sum_{i,j=1}^{\infty}|\left<h_{\theta},\psi_i^{\theta}\right>_{\theta}\left<h_{\theta},\psi_j^{\theta}\right>_{\theta}|C_{i,j}(T,\theta)<\infty\ ,
\]
where $C_{i,j}(T,\theta)$ is given by (\ref{Eq:BoundedTerm0}). Denote by
$u^{2}_{\theta}(h_{\theta})\doteq \sum_{i,j=1}^{\infty}\frac{\left<h_{\theta},\psi_i^{\theta}\right>_{\theta}\left<h_{\theta},\psi_j^{\theta}\right>_{\theta}\pi_0^{\theta}[\psi_i^{\theta}]\pi_0^{\theta}[\psi_j^{\theta}]}{\lambda_i^{\theta} +\lambda_j^{\theta}}<\infty $
and
$v^{2}_{\theta}(h_{\theta})\doteq \sum_{i,j=1}^{\infty}\frac{\left|\left<h_{\theta},\psi_i^{\theta}\right>_{\theta}\left<h_{\theta},\psi_j^{\theta}\right>_{\theta}\right|^{2}}{\lambda_i^{\theta}+\lambda_j^{\theta}}$, with $v_\theta^2(h_\theta)<\infty$ by Parseval's identity (see Remark \ref{R:AssumptionOnVariance}). Then, and under $\mathbb P_{\alpha}$, and for any fixed $t\in(0,T]$ we have
\[\frac{1}{\sqrt\delta}\left(\log\left( \phi^{\delta,\theta}_{t}[1]\right)-\log\left( \bar\phi^{\delta,\theta}_{t}[1]\right)\right)\Rightarrow\mathcal W\left(u^{2}_{\theta}(h_{\theta})+t v^{2}_{\theta}(h_{\theta})\right)\qquad\hbox{as }\delta\rightarrow 0 \]
in distribution, where $\mathcal{W}\left(u^{2}_{\theta}(h_{\theta})+t v^{2}_{\theta}(h_{\theta})\right)$ is a normal random variable with mean zero and variance $u^{2}_{\theta}(h_{\theta})+t v^{2}_{\theta}(h_{\theta})$.
\end{theorem}

If $X_0$ starts in its invariant distribution, then $C(\psi_{i},T,\theta) = 1$ and $\pi^{\theta}_{0}[\psi_{i}^{\theta}]=0$ for all $i\geq 0$ and $C_{i,j}(T,\theta) = \frac{2}{\lambda_i^{\theta}}+\frac{2}{\lambda_j^{\theta}}$ for all $i,j\geq 0$, and we have the following corollary from Theorem \ref{thm:CLT}:
\begin{corollary}\textbf{(Likelihood CLT for Paths).}\label{C:CLT}
Assume Conditions \ref{A:Assumption1} and \ref{A:Assumption3}. Moreover, assume that  for all $\theta\in\Theta$
\[
\sum_{i,j=1}^{\infty}|\left<h_{\theta},\psi_i^{\theta}\right>_{\theta}\left<h_{\theta},\psi_j^{\theta}\right>_{\theta}|\left(\frac{1}{\lambda_i^{\theta}}+\frac{1}{\lambda_j^{\theta}}\right)<\infty\ .
\]
If $X_{0}$ is distributed according to the invariant measure $\mu_\theta$ (i.e. $\pi_0^\theta[f] = \bar f_\theta$ for all $f\in\mathcal A_\eta^\theta$ and any $\theta\in\Theta$), then under $\mathbb P_\alpha$ we have
\[\frac{1}{\sqrt\delta}\left(\log\left( \phi^{\delta,\theta}_{\cdot}[1]\right)-\log\left( \bar\phi^{\delta,\theta}_{\cdot}[1]\right)\right)\Rightarrow\sqrt{v^{2}_{\theta}(h_{\theta})}\mathcal W\left(\cdot\right)\qquad\hbox{as }\delta\rightarrow 0 \]
in distribution on $C([0,T],\mathbb R)$, where $\mathcal W$ is a Brownian motion, and $v_\theta^2(h_\theta)$ is as defined in Theorem \ref{thm:CLT}.
\end{corollary}

Before continuing with the proofs of the Theorem \ref{thm:CLT} and Corollary \ref{C:CLT}, we make some remarks related to the conditions that appear in the statement of the CLT.

\begin{remark}\label{R:AssumptionOnVariance}
The orthonormal basis of eigenfunctions that was assumed in Condition \ref{A:Assumption3} is enough to ensure that the variance $v_\theta^2(h_\theta)<\infty$. Indeed, by Parseval's identity we have
\begin{align*}v^{2}_{\theta}(h_{\theta})&= \sum_{i,j=1}^{\infty}\frac{\left|\left<h_{\theta},\psi_i^{\theta}\right>_{\theta}\left<h_{\theta},\psi_j^{\theta}\right>_{\theta}\right|^{2}}{\lambda_i^{\theta}+\lambda_j^{\theta}}\leq \frac{1}{2\lambda_1^\theta}\sum_{i,j=1}^{\infty}\left|\left<h_{\theta},\psi_i^{\theta}\right>_{\theta}\left<h_{\theta},\psi_j^{\theta}\right>_{\theta}\right|^{2}\\
& = \frac{1}{2\lambda_1^\theta}\left(\sum_{i=1}^{\infty}\left|\left<h_{\theta},\psi_i^{\theta}\right>_{\theta}\right|^{2}\right)^2=\frac{1}{2\lambda_1^\theta}\left(\int_{\mathcal{X}} |h_\theta(x)|^2\mu_\theta(dx)\right)^2 = \frac{\left\|h_{\theta}\right\|^{4}_{L^{2}(\mathcal{X},\mu_{\theta})}}{2\lambda_1^\theta}<\frac{C_h^4}{2\lambda_1^\theta}<\infty \ ,
\end{align*}
where $C_h$ is the constant from Condition \ref{A:Assumption3}. Finiteness of $u_\theta^2(h_\theta)$ follows from equation \eqref{Eq:uthetaLimit} in the proof of Theorem \ref{thm:CLT}.
\end{remark}

\begin{remark}\textbf{(Absolutely Summable $\mathbf{h_\theta}$).}
The function $h_\theta(x)$ is said to be an \textit{absolutely summable} function if
\[\sum_{i=0}^\infty|\left<h_\theta,\psi_i^\theta\right>_\theta|<\infty\ .\]
Absolute summability is sufficient for Corollary \ref{C:CLT} to hold. Indeed, notice that
\[
\sum_{i,j=1}^{\infty}|\left<h_{\theta},\psi_i^{\theta}\right>_{\theta}\left<h_{\theta},\psi_j^{\theta}\right>_{\theta}|\left(\frac{1}{\lambda_i^{\theta}}+\frac{1}{\lambda_j^{\theta}}\right)\leq \frac{2}{\lambda_1^{\theta}} \left(\sum_{i=1}^{\infty}|\left<h_{\theta},\psi_i^{\theta}\right>_{\theta}|\right)^{2}.
\]
A similar treatment applies to the more general summability constraint that appears in Theorem \ref{thm:CLT}. For more on functions whose eigen-coefficients decay fast enough to ensure absolute convergence, see the conditions/examples given in \cite{boyd2000,boyd1984}.
\end{remark}

\begin{remark}\textbf{(Converging Initial Distributions).}
Corollary \ref{C:CLT} could be generalized to the case where the initial distribution depends on $\delta$ and converges to the invariant distribution. That is, assuming a priori the limit
\[\pi_0^{\delta,\theta}[f]\rightarrow \bar f_\theta\qquad\hbox{as $\delta\rightarrow 0$, $\forall \theta\in\Theta$}\ ,\forall f\in\mathcal A_\eta^\theta\ ,\]
then one expects that the same path-wise limit remains as stated in the corollary. However, generalization of the proofs in this paper will require verification that the initial filters $\pi_0^{\delta,\theta}$ satisfy equation \eqref{Eq:BoundedTerm01} and allow for the limit to pass into the sum in equation \eqref{Eq:uthetaLimit}.
\end{remark}

\begin{remark}
We could also combine Theorem \ref{thm:CLT} and Corollary \ref{C:CLT} by writing
\[\frac{1}{\sqrt\delta}\left(\log\left( \phi^{\delta,\theta}_{\cdot}[1]\right)-\log\left( \bar\phi^{\delta,\theta}_{\cdot}[1]\right)-\sqrt \delta R_\cdot^{1,\delta}\right)\Rightarrow\sqrt{v^{2}_{\theta}(h_{\theta})}\mathcal W\left(\cdot\right)\qquad\hbox{as }\delta\rightarrow 0 \]
in $C([0,T],\mathbb R)$, where $R_t^{1,\delta}$ is given by equation \eqref{Eq:ForLastTerm} in the proof of Theorem \ref{thm:CLT}.
\end{remark}

\begin{proof} [Proof of Theorem \ref{thm:CLT}] The proof of this theorem involves showing that $(\dagger)$ and $(**)$ from equation \eqref{eq:logLderiv} converge to zero in probability uniformly in $t\in[0,T]$, and then showing that $(*)$ converges weakly to the appropriate normal distribution. Then, the result follows by Slutzky's theorem (see \cite{Billingsley}).

First we consider the term $(\dagger)$. By Lemma \ref{L:UsedForTightness} we have that there exists a constant $C<\infty$ such that
\[
\sup_{\delta\in(0,1)}\mathbb{E}_{\alpha}\sup_{t\in[0,T]}\int_{0}^{t}\left[\frac{1}{\delta}\left|\pi_s^{\delta,\theta}[\tilde h_{\theta}]\right|^{2}\right]ds
\leq\sup_{\delta\in(0,1)}\mathbb{E}_{\alpha}\int_{0}^{T}\left[\frac{1}{\delta}\left|\pi_s^{\delta,\theta}[\tilde h_{\theta}]\right|^{2}\right]ds<C
\]
Therefore, the conclusion
\[
\lim_{\delta\downarrow 0}\mathbb{E}_{\alpha}\frac{1}{2\sqrt\delta}\sup_{t\in[0,T]}\int_0^t\left| \pi_s^{\delta,\theta}[\tilde h_{\theta}]\right|^2ds=0
\]
follows, implying the claimed convergence of the term  $(\dagger)$ in $\mathbb{P}_{\alpha}$-probability, uniformly in $t\in[0,T]$. Convergence to zero in $\mathbb{P}_{\alpha}$-probability of the $(**)$ term follows by Lemma \ref{L:MismatchTerm}.

Now we turn our attention toward $(*)$, and define the integrated process,
\begin{align*}
I_{t}^{\delta} &\doteq\int_0^t\frac{1}{\sqrt\delta}\left( \pi_s^{\delta,\theta}[\tilde h_{\theta}]\right)d\nu_s^{\delta,\alpha}\ ,
\end{align*}
which is a $\mathbb P_\alpha$ martingale. Since $h_{\theta}$ is bounded, we clearly have that $\tilde{h}_{\theta}\in L^{2}(\mathcal{X},\mu_{\theta})$ and hence, we have the representation
\[
\tilde{h}_{\theta}(x)=\sum_{i=0}^{\infty}\left<\tilde{h}_{\theta},\psi_i^{\theta}\right>_{\theta} \psi_i^{\theta}(x)=\sum_{i=1}^{\infty}\left<h_{\theta},\psi_i^{\theta}\right>_{\theta} \psi_i^{\theta}(x).
\]
Thus, we get
\begin{equation}
\frac{1}{\sqrt\delta} \pi_s^{\delta,\theta}[\tilde h_{\theta}]= \sum_{i=1}^{\infty}\left<h_{\theta},\psi_i^{\theta}\right>_{\theta} \frac{1}{\sqrt\delta} \pi_s^{\delta,\theta}[\psi_i^{\theta}]
\end{equation}
From this and equation (\ref{eq:KSsoln2}), it follows that
\begin{align}
\frac{1}{\sqrt\delta}\pi_s^{\delta,\theta}[\tilde h_{\theta}]&=
\frac{1}{\sqrt\delta}\sum_{i=1}^{\infty}\left<h_{\theta},\psi_i^{\theta}\right>_{\theta}e^{-\frac{\lambda_i^{\theta} s}{\delta}}\pi_0^{\theta}[\psi_i^{\theta}]\nonumber\\
&\hspace{0.2cm}+\frac{1}{\sqrt\delta}\int_0^s \sum_{i=1}^{\infty}\left<h_{\theta},\psi_i^{\theta}\right>_{\theta} e^{-\frac{\lambda_i^{\theta}(s-\rho)}{\delta}}\left(\pi_\rho^{\delta,\theta}[h_{\theta}\psi_i^{\theta}]-\left<h_{\theta},\psi_i^{\theta}\right>_{\theta}
-\pi_\rho^{\delta,\theta}[h_{\theta}]\pi_\rho^{\delta,\theta}[\psi_i^{\theta}]
\right)d\nu_\rho^{\delta,\theta}\nonumber\\
&\hspace{0.2cm}+\sum_{i=1}^{\infty}\left<h_{\theta},\psi_i^{\theta}\right>_{\theta}\frac{\left<h_{\theta},\psi_i^{\theta}\right>_{\theta}}{\sqrt\delta}\int_0^s e^{-\frac{\lambda_i^{\theta}(s-\rho)}{\delta}}\left(\pi_\rho^{\delta,\theta}[h_{\theta}]-\pi_\rho^{\delta,\alpha}[h_{\alpha}]\right)d\rho\nonumber\\
&\hspace{0.2cm}+\sum_{i=1}^{\infty}\left<h_{\theta},\psi_i^{\theta}\right>_{\theta}\frac{\left<h_{\theta},\psi_i^{\theta}\right>_{\theta}}{\sqrt\delta}\int_0^s e^{-\frac{\lambda_i^{\theta}(s-\rho)}{\delta}}d\nu_\rho^{\delta,\alpha}\label{Eq:ForLastTerm0}
\end{align}

Hence, we have
\begin{align}
I^{\delta}_{t}&=
\sum_{i=1}^{\infty}\left<h_{\theta},\psi_i^{\theta}\right>_{\theta}\pi_0^{\theta}[\psi_i^{\theta}] \frac{1}{\sqrt\delta}\int_{0}^{t}e^{-\frac{\lambda_i^{\theta} s}{\delta}}d\nu_s^{\delta,\alpha}\nonumber\\
&\hspace{0.2cm}+ \sum_{i=1}^{\infty}\left<h_{\theta},\psi_i^{\theta}\right>_{\theta}  \frac{1}{\sqrt\delta}\int_{0}^{t}\left[\int_0^s e^{-\frac{\lambda_i^{\theta}(s-\rho)}{\delta}}\left(\pi_\rho^{\delta,\theta}[h_{\theta}\psi_i^{\theta}]-\left<h_{\theta},\psi_i^{\theta}\right>_{\theta}
-\pi_\rho^{\delta,\theta}[h_{\theta}]\pi_\rho^{\delta,\theta}[\psi_i^{\theta}]
\right)d\nu_\rho^{\delta,\theta}\right]d\nu_s^{\delta,\alpha}\nonumber\\
&\hspace{0.2cm}+\sum_{i=1}^{\infty}\left<h_{\theta},\psi_i^{\theta}\right>_{\theta}\frac{\left<h_{\theta},\psi_i^{\theta}\right>_{\theta}}{\sqrt\delta}\int_{0}^{t}\left[\int_0^s e^{-\frac{\lambda_i^{\theta}(s-\rho)}{\delta}}\left(\pi_\rho^{\delta,\theta}[h_{\theta}]-\pi_\rho^{\delta,\alpha}[h_{\alpha}]\right)d\rho\right]d\nu_s^{\delta,\alpha}\nonumber\\
&\hspace{0.2cm}+\sum_{i=1}^{\infty}\left<h_{\theta},\psi_i^{\theta}\right>_{\theta}\frac{\left<h_{\theta},\psi_i^{\theta}\right>_{\theta}}{\sqrt\delta}\int_0^t \left[\int_0^s e^{-\frac{\lambda_i^{\theta}(s-\rho)}{\delta}}d\nu_\rho^{\delta,\alpha}\right]d\nu_s^{\delta,\alpha} \nonumber\\
&=R^{1,\delta}_{t}+R^{2,\delta}_{t}+R^{3,\delta}_{t}+R^{4,\delta}_{t}\label{Eq:ForLastTerm}
\end{align}
where $R^{j,\delta}_{t}$ for $j=1,2,3,4$ are defined by the four lines in (\ref{Eq:ForLastTerm}). We treat each of the $R^{j,\delta}_{t}$ terms separately. By Lemmas \ref{L:R2convergence} and \ref{L:R3convergence}, we have that
\begin{equation*}
\lim_{\delta\downarrow 0}\left\{\mathbb{E}_{\alpha}\sup_{t\in[0,T]}\left|R^{2,\delta}_{t}\right|^{2}+\mathbb{E}_{\alpha}\sup_{t\in[0,T]}\left|R^{3,\delta}_{t}\right|^{2}\right\}=0.
\end{equation*}
Thus, we have established that uniformly in $t\in[0,T]$
\begin{equation}
I^{\delta}_{t}-(R^{1,\delta}_{t}+R^{4,\delta}_{t})\rightarrow 0, \text{ in } \mathbb{P}_{\alpha}\text{ probability as }\delta\downarrow 0\label{Eq:EquivalentWeakLimit}
\end{equation}

It remains to treat the first and the last term, i.e., the term $R^{1,\delta}_{t}$ and the term $R^{4,\delta}_{t}$. Recall that
\begin{align}
R^{1,\delta}_{t}+R^{4,\delta}_{t}&=\int_{0}^{t}\left( \sum_{i=1}^{\infty}\left<h_{\theta},\psi_i^{\theta}\right>_{\theta} \frac{1}{\sqrt\delta}e^{-\frac{\lambda_i^{\theta} s}{\delta}}\pi_0^{\theta}[\psi_i^{\theta}]+\sum_{i=1}^{\infty}\left<h_{\theta},\psi_i^{\theta}\right>_{\theta}^{2}\frac{1}{\sqrt\delta} \int_0^s e^{-\frac{\lambda_i^{\theta}(s-\rho)}{\delta}}d\nu_\rho^{\delta,\alpha}\right)d\nu_s^{\delta,\alpha}.
\end{align}
The solution to the linear SDE
\begin{equation}
\Xi^{\delta,i}_{t}=-\frac{\lambda_{i}^{\theta}}{\delta}\int_{0}^{t}\Xi^{\delta,i}_{s}ds+\frac{1}{\sqrt{\delta}}\nu_t^{\delta,\alpha}\label{Eq:OUSDE1}
\end{equation}
is simply
\begin{equation}
\Xi^{\delta,i}_{t}=\frac{1}{\sqrt\delta}\int_0^t e^{-\frac{\lambda_i^{\theta}(t-s)}{\delta}}d\nu_s^{\delta,\alpha}\label{Eq:OUSDE2}
\end{equation}
So, by the martingale representation theorem, there is an appropriate Wiener process $\mathcal{W}$ such that we have in distribution (see Theorem 4.6 on page 174 \cite{KaratzasShreve})
\begin{align}
&R^{1,\delta}_{t}+R^{4,\delta}_{t}=\int_{0}^{t}\left( \sum_{i=1}^{\infty}\left<h_{\theta},\psi_i^{\theta}\right>_{\theta} \frac{1}{\sqrt\delta}e^{-\frac{\lambda_i^{\theta} s}{\delta}}\pi_0^{\theta}[\psi_i^{\theta}]+\sum_{i=1}^{\infty}\left<h_{\theta},\psi_i^{\theta}\right>_{\theta}^{2}\Xi^{\delta,i}_{s}\right)d\nu_s^{\delta,\alpha}\nonumber\\
&=\mathcal{W}\left(\int_{0}^{t}\left( \sum_{i=1}^{\infty}\left<h_{\theta},\psi_i^{\theta}\right>_{\theta} \frac{1}{\sqrt\delta}e^{-\frac{\lambda_i^{\theta} s}{\delta}}\pi_0^{\theta}[\psi_i^{\theta}]+\sum_{i=1}^{\infty}\left<h_{\theta},\psi_i^{\theta}\right>_{\theta}^{2}\Xi^{\delta,i}_{s}\right)^{2}ds\right)\nonumber\\
&=\mathcal{W}\left(\int_{0}^{t}\left( \sum_{i,j=1}^{\infty}\left<h_{\theta},\psi_i^{\theta}\right>_{\theta}\left<h_{\theta},\psi_j^{\theta}\right>_{\theta} \frac{1}{\delta}e^{-\frac{\left(\lambda_i^{\theta} +\lambda_j^{\theta}\right) s}{\delta}}\pi_0^{\theta}[\psi_i^{\theta}]\pi_0^{\theta}[\psi_j^{\theta}]+
\sum_{i,j=1}^{\infty}\left|\left<h_{\theta},\psi_i^{\theta}\right>_{\theta}\left<h_{\theta},\psi_j^{\theta}\right>_{\theta}\right|^{2}\Xi^{\delta,i}_{s}\Xi^{\delta,j}_{s}\right.\right.\nonumber\\
&\hspace{1cm}\left.\left.+
2\sum_{i,j=1}^{\infty}\left<h_{\theta},\psi_i^{\theta}\right>_{\theta}\left|\left<h_{\theta},\psi_j^{\theta}\right>_{\theta}\right|^{2}\frac{1}{\sqrt\delta}e^{-\frac{\lambda_i^{\theta} s}{\delta}}\pi_0^{\theta}[\psi_i^{\theta}]\Xi^{\delta,j}_{s}\right)ds\right)\nonumber\\
&=\mathcal{W}\left( \sum_{i,j=1}^{\infty}\frac{\left<h_{\theta},\psi_i^{\theta}\right>_{\theta}\left<h_{\theta},\psi_j^{\theta}\right>_{\theta}}{\lambda_i^{\theta} +\lambda_j^{\theta}} \pi_0^{\theta}[\psi_i^{\theta}]\pi_0^{\theta}[\psi_j^{\theta}]\left(1-e^{-\frac{\left(\lambda_i^{\theta}+\lambda_j^{\theta}\right) t}{\delta}}\right)+
\sum_{i,j=1}^{\infty}\left|\left<h_{\theta},\psi_i^{\theta}\right>_{\theta}\left<h_{\theta},\psi_j^{\theta}\right>_{\theta}\right|^{2}\int_{0}^{t}\Xi^{\delta,i}_{s}\Xi^{\delta,j}_{s}ds\right.\nonumber\\
&\hspace{1cm}\left.+
2\sum_{i,j=1}^{\infty}\left<h_{\theta},\psi_i^{\theta}\right>_{\theta}\left|\left<h_{\theta},\psi_j^{\theta}\right>_{\theta}\right|^{2}\pi_0^{\theta}[\psi_i^{\theta}]\frac{1}{\sqrt\delta}\int_{0}^{t}e^{-\frac{\lambda_i^{\theta} s}{\delta}}\Xi^{\delta,j}_{s}ds\right)\nonumber\\
&=\mathcal{W}\left( J^{1,\delta}_{t}+J^{2,\delta}_{t}+J^{3,\delta}_{t}\right)\ ,
\label{Eq:sum3Js}
\end{align}
where $J^{\ell,\delta}_{t}$ is the $\ell^{th}$ term in the variance of $\mathcal{W}(\cdot)$. So in order to find where $R^{1,\delta}_{t}+R^{4,\delta}_{t}$ converges to in distribution, we need to find the limit in probability of $J^{1,\delta}_{t}+J^{2,\delta}_{t}+J^{3,\delta}_{t}$. For each fixed  $t\in(0,T]$ we have
\begin{align}
\lim_{\delta\downarrow 0}J^{1,\delta}_{t}&=\lim_{\delta\downarrow 0}\sum_{i,j=1}^{\infty}\frac{\left<h_{\theta},\psi_i^{\theta}\right>_{\theta}\left<h_{\theta},\psi_j^{\theta}\right>_{\theta}}{\lambda_i^{\theta} +\lambda_j^{\theta}} \pi_0^{\theta}[\psi_i^{\theta}]\pi_0^{\theta}[\psi_j^{\theta}]\left(1-e^{-\frac{\left(\lambda_i^{\theta}+\lambda_j^{\theta}\right) t}{\delta}}\right)\nonumber\\
&=\sum_{i,j=1}^{\infty}\frac{\left<h_{\theta},\psi_i^{\theta}\right>_{\theta}\left<h_{\theta},\psi_j^{\theta}\right>_{\theta}}{\lambda_i^{\theta} +\lambda_j^{\theta}} \pi_0^{\theta}[\psi_i^{\theta}]\pi_0^{\theta}[\psi_j^{\theta}]=u^{2}_{\theta}(h_{\theta})\ .
\label{Eq:uthetaLimit}
\end{align}

For $J^{2,\delta}_{t}$ we use ergodicity of the pair $(\Xi^{\delta,i}_{t}, \Xi^{\delta,j}_{t})$. Clearly, for any $i\geq 1$, $\Xi^{\delta,i}_{t}$ is ergodic (it is a one-dimensional Ornstein-Uhlenbeck process). Also, one can check the Fokker-Planck equation for the pair $(\Xi^{\delta,i}_{t}, \Xi^{\delta,j}_{t})$ to see that for $\lambda_{i}^{\theta}\neq \lambda_{j}^{\theta}$, $(\Xi^{\delta,i}_{t}, \Xi^{\delta,j}_{t})=_d(\Xi^{1,i}_{t/\delta}, \Xi^{1,j}_{t/\delta})$ is jointly Gaussian and ergodic, and for every $t\in[0,T]$ converges as $\delta\downarrow 0$ in distribution to  a pair of jointly Gaussian random variables $(Z_{i},Z_{j})$ with mean zero and invertible covariance matrix. Thus, by the ergodic theorem we have for every $t\geq 0$

\begin{equation}
\lim_{\delta\downarrow 0}\mathbb{E}\left|\int_{0}^{t}\Xi^{\delta,i}_{s}\Xi^{\delta,j}_{s}ds- t\beta^{i,j}\right|=\lim_{\delta\downarrow 0}\mathbb{E}\left|\delta\int_{0}^{t/\delta}\Xi^{i,1}_{s}\Xi^{j,1}_{s}ds- t\beta^{i,j}\right|=0
\end{equation}
where $\beta^{i,j}=\mathbb{E}[Z_{i}Z_{j}]=\lim_{\delta\downarrow 0}\mathbb{E}\left[\Xi^{\delta,i}_{t}\Xi^{\delta,j}_{t}\right]=\frac{1}{\lambda_i^{\theta}+\lambda_j^{\theta}}$. Since by assumption we have $v^{2}_{\theta}(h_{\theta})<\infty$, we have for every $t\geq 0$
\begin{equation}
\lim_{\delta\downarrow 0}J^{2,\delta}_{t}=t ~v^{2}_{\theta}(h_{\theta}), \quad \text{in probability as }\delta\downarrow 0\ .
\end{equation}
For similar reasons, we also obtain that for every $t\geq 0$
\begin{equation}
\lim_{\delta\downarrow 0}J^{3,\delta}_{t}=0, \quad \text{in probability as }\delta\downarrow 0\ .
\label{Eq:J3limit}
\end{equation}
Hence, we get that for every fixed $t\in(0,T]$
\begin{equation}
R^{1,\delta}_{t}+R^{4,\delta}_{t}\Rightarrow  \mathcal{W}\left(u^{2}_{\theta}(h_{\theta})\indicator{t>0}+ t v^{2}_{\theta}(h_{\theta})\right)
\end{equation}
as $\delta\rightarrow 0$, which then implies that for every fixed $t\in(0,T]$
\begin{equation}
I^{\delta}_{t}\Rightarrow  \mathcal{W}\left(u^{2}_{\theta}(h_{\theta})+ t v^{2}_{\theta}(h_{\theta})\right)
\end{equation}
as $\delta\rightarrow 0$.
\end{proof}

\begin{proof}[Proof of Corollary \ref{C:CLT}]
The proof of the CLT for paths requires identification of the weak limit of $I_\cdot^{\delta}$, which we do using the martingale central limit theorem that is stated in Theorem 1.4 on page 339 of \cite{EthierKurtz1986}. In particular, the process $I_t^\delta$ is a martingale and takes values in the space $C([0,T];\mathbb R)$ with probability one, so it follows that
\begin{equation}
\label{eq:condForMartinCLT}\lim_{\delta\downarrow0}\mathbb E_\alpha\left[\sup_{t\in[0,T]}\left|I_t^\delta-I_{t-}^\delta\right|\right] = 0\ .
\end{equation}
Given \eqref{eq:condForMartinCLT}, if the quadratic variation of $I^\delta_{t}$ converges to a constant multiple of $t$ for each $t\in[0,T]$, then the martingale CLT says that $I_\cdot^\delta$ converges weakly to a Brownian motion multiplied by the limiting quadratic variation.

Convergence of the quadratic variation was shown in the proof of Theorem \ref{thm:CLT} by showing that terms $J_t^{1,\delta}$, $J_t^{2,\delta}$, and $J_t^{3,\delta}$ converge in probability as $\delta\rightarrow 0$. Indeed, if $X_{0}$ follows the invariant distribution, then $\pi_0^{\theta}[\psi_i^{\theta}]=0$ and $C(\psi_{i},T,\theta)=1$ for all $i\in\mathbb{N}$,. This means that
\[
C_{i,j}(T,\theta)=\frac{2}{\lambda_i^{\theta}}+\frac{2}{\lambda_j^{\theta}}
\]
and that the terms $J^{1,\delta}_{t}=J^{3,\delta}_{t}=0$ in the proof of Theorem \ref{thm:CLT}. Hence, the quadratic variation of $I_t^\delta$ converges to $tv_\theta(h_\theta)$ in $\mathbb P_\alpha$-probability for all $t\in [0,T]$, and so we get that in distribution
\begin{equation}
I^{\delta}_{\cdot}\Rightarrow  \sqrt{v^{2}_{\theta}(h_{\theta})}\mathcal{W}(\cdot)\qquad\hbox{under $\mathbb P_\alpha$}\ .
\end{equation}

The remaining terms $(**)$ and $(\dagger)$ from equation \eqref{eq:logLderiv} were shown in the proof of Theorem \ref{thm:CLT} to go to zero in $\mathbb P_\alpha$-probability
uniformly for all $t\in[0,T]$, and therefore they also (both) converge pathwise to zero in probability. Hence, all three terms in equation \eqref{eq:logLderiv} converge pathwise,
two of which in probability to zero, and the other weakly to a $\sqrt{v_\theta^2(h_\theta)}\mathcal W(\cdot)$. Therefore, by Slutzky's theorem the sum of all three terms converges
weakly to $\sqrt{v_\theta^2(h_\theta)}\mathcal W(\cdot)$.
\end{proof}


\section{On Statistical Inference}\label{S:ParameterEstimation}
In Subsection \ref{SS:filterLimit}, and in particular in Corollary \ref{Lemma:likelihoodConvergence}, we proved that the likelihood function $\phi^{\delta,\theta}_{T}[1]$ is close in
probability to the reduced likelihood $\bar{\phi}^{\delta,\theta}_{T}[1]$ when $\delta$ is small. In this section, we use these results to do statistical inference for the unknown true parameter $\alpha\in\Theta$  based on the MLE of the log-likelihood function.

Corollary \ref{Lemma:likelihoodConvergence} suggests that for parameter estimation, we can approximate the log-likelihood
\begin{equation}
\rho_T^\delta(\theta)=\log \phi^{\delta,\theta}_{T}[1]=\log\mathbb{E}_{\theta}^{*}\left[ Z_T^{\delta,\theta} \Big |\mathcal{Y}^{\delta}_T\right]\label{Eq:LogLikelihood}
\end{equation}
by the `reduced' log-likelihood
\begin{equation}
\bar{\rho}_T^\delta(\theta)=\log \bar{\phi}^{\delta,\theta}_{T}[1]=\log \mathbb{E}_{\theta}^{*}\left[ \bar{Z}_T^{\delta,\theta} \Big |\mathcal{Y}^{\delta}_T\right]=\bar h_{\theta} Y_T^\delta-\frac 12|\bar h_{\theta}|^2 T.\label{Eq:IntermediateLikelihood}
\end{equation}
Clearly, $\bar{\rho}_T^\delta(\theta)$ is of reduced dimension and easier to work with, as long as one can compute or approximate the invariant measure of the fast dynamics and thus compute or approximate $\bar h_{\theta}$.
Based on the full log-likelihood (\ref{Eq:LogLikelihood}), one would need to compute $\rho_T^\delta(\theta)$ and thus rely on methods such as particle filters or sequential Monte Carlo (e.g., Chapter 9 of \cite{bainCrisan}). However, such methods can be computational expensive due to high-dimensionality issues.

With this mind, we prove that the MLE based on (\ref{Eq:IntermediateLikelihood}) is in fact, under the appropriate identifiability condition, asymptotically consistent when the time horizon is large enough.

\begin{condition}\label{A:Assumption4}
\begin{enumerate}
\item The mapping $\bar h_\theta$ from $\Theta\mapsto \mathbb R^m$ is a one-to-one function of $\theta$.
\item
There are constants $C>0$, $p\geq 1$ and $q>1$, such that for any $\theta_{1},\theta_{2}\in\Theta$,
\[
|\bar{h}_{\theta_{1}}-\bar{h}_{\theta_{2}}|^{2p}\leq C|\theta_{1}-\theta_{2}|^{q}.
\]
\end{enumerate}
\end{condition}
Recall the definition of MLE from equation \eqref{Eq:MLE1}, and let us equivalently define the reduced estimator as
\begin{equation}
\bar{\theta}^{\delta}_{T}\doteq\argmax_{\theta\in\Theta}\bar{\rho}_T^\delta(\theta).\label{Eq:ThetaMLE}
\end{equation}
Continuity of $\bar{\rho}_T^\delta(\cdot)$ that is ensured by Condition \ref{A:Assumption4}, together with compactness of $\Theta$, imply that the corresponding maximizer exists almost surely.

Next, we prove consistency of the reduced log-likelihood.
\begin{theorem}\label{T:ConsisitencyReducedLikelihood}
Assume Conditions \ref{A:Assumption1} and \ref{A:Assumption4}. Let $\alpha$ be the true parameter value. Let us denote by $\bar{\Theta}^{\delta}_{T}$ the equivalence class of maximizers of $\bar{\rho}_T^\delta(\theta)$. The maximum likelihood estimator based on (\ref{Eq:IntermediateLikelihood}), i.e., any $\bar{\theta}^{\delta}_{T}\in \bar{\Theta}^{\delta}_{T}$, is strongly consistent as
first $\delta\downarrow 0$ and then $T\rightarrow\infty$, i.e., for any $\varepsilon>0$
\[
\lim_{T\rightarrow\infty}\lim_{\delta\downarrow 0}\mathbb{P}_{\alpha}\left(\left|\bar{\theta}^{\delta}_{T}-\alpha\right|>\varepsilon\right)=0.
\]
\end{theorem}
\begin{proof}
Let us denote
\[
\bar{\rho}_T^\delta(\theta,\alpha)=\bar h_{\theta} \int_{0}^{T}h_{\alpha}(X^{\delta}_{s})ds +\bar h_{\theta} W_{T}-\frac 12|\bar h_{\theta}|^2T.
\]
Then, we have
\begin{align}
\mathbb{E}_{\alpha}\left|\bar{\rho}_T^\delta(\theta_{1},\alpha)-\bar{\rho}_T^\delta(\theta_{2},\alpha) \right|^{2p}&\leq C |\bar{h}_{\theta_{1}}-\bar{h}_{\theta_{2}}|^{2p}
\left(1+\mathbb{E}_{\alpha}\int_{0}^{T}|h_{\alpha}(X^{\delta}_{s})|^{2p}ds \right),\nonumber\\
&\leq C |\theta_{1}-\theta_{2}|^{q}\nonumber
\end{align}
where we used Condition \ref{A:Assumption4}. The constant $C$ might change from line to line, but we do not indicate this in the notation.
Next, the ergodic theorem guarantees that the finite dimensional distributions of $\bar{\rho}^{\delta}_{T}(\cdot,\alpha)$ converge with probability $1$, as $\delta\downarrow 0$, to those of
\[
\bar{\rho}_T(\theta,\alpha)=\bar h_{\theta} \bar h_{\alpha}T +\bar h_{\theta} W_{T}-\frac 12|\bar h_{\theta}|^2T= -\frac{1}{2}|\bar h_{\theta} -\bar h_{\alpha}|^{2}T+\frac{1}{2}|\bar h_{\alpha}|^2T+ \bar{h}_{\theta}W_{T}.
\]
Therefore, by Theorem 12.3 in \cite{Billingsley}, we have weak convergence of the measure $\bar{\rho}^{\delta}_{T}(\cdot,\alpha)$ to that of $\bar{\rho}_{T}(\cdot,\alpha)$. Hence, we have obtained (in a similar manner to Theorem 2.25 on page 161 of \cite{Kutoyants})
\begin{align}
\lim_{\delta\downarrow 0}\mathbb{P}_{\alpha}\left(\left|\bar{\theta}^{\delta}_{T}-\alpha\right|>\varepsilon\right)&=
\lim_{\delta\downarrow 0}\mathbb{P}_{\alpha}\left(\sup_{\left|\theta-\alpha\right|>\varepsilon}\frac{1}{T}\bar{\rho}_{T}^\delta(\theta,\alpha)>\sup_{\left|\theta-\alpha\right|\leq\varepsilon}\frac{1}{T}\bar{\rho}_{T}^\delta(\theta,\alpha)\right)\nonumber\\
&=
\mathbb{P}_{\alpha}\left(\sup_{\left|\theta-\alpha\right|>\varepsilon}\frac{1}{T}\bar{\rho}_{T}(\theta,\alpha)>\sup_{\left|\theta-\alpha\right|\leq\varepsilon}\frac{1}{T}\bar{\rho}_{T}(\theta,\alpha)\right)\nonumber
\end{align}
Hence, if we now define $\bar{\bar{\rho}}(\theta,\alpha)=-\frac{1}{2}|\bar h_{\theta} -\bar h_{\alpha}|^{2}+\frac{1}{2}|\bar h_{\alpha}|^2$, we then get
\begin{align}
\lim_{T\rightarrow\infty}\lim_{\delta\downarrow 0}\mathbb{P}_{\alpha}\left(\left|\bar{\theta}^{\delta}_{T}-\alpha\right|>\varepsilon\right)&=
\indicator{\sup_{\left|\theta-\alpha\right|>\varepsilon}\bar{\bar{\rho}}(\theta,\alpha)>\sup_{\left|\theta-\alpha\right|\leq\varepsilon}\bar{\bar{\rho}}(\theta,\alpha)}=0\ ,\nonumber
\end{align}
where the last computation used the fact that $\bar{\bar{\rho}}(\theta,\alpha)$ has a unique maximum at $\theta=\alpha$, which follows from part i) of Condition \ref{A:Assumption4}. With this, we conclude the proof of the theorem.
\end{proof}

Solving the equation $\frac{\partial}{\partial \theta}\bar{\rho}^{\delta}_{T}( \theta)=0$ for $\theta\in\Theta$, we define $\tilde{\theta}^{\delta}_{T}$ to be the solution (if it exists) to
\begin{align}
 \bar{h}_{\theta}=\frac 1TY^{\delta}_{T}.\label{Eq:EquationForReducedMLE}
\end{align}
It is clear that (\ref{Eq:ThetaMLE}) and (\ref{Eq:EquationForReducedMLE}) are not equivalent; (\ref{Eq:ThetaMLE}) contains all local minima and local maxima of $\bar{\rho}^{\delta}_{T}(\theta)$ which may be more than one. Also equation \eqref{Eq:EquationForReducedMLE} may not even have a solution in $\Theta$ with positive probability. For example, letting $\tilde{\theta}^{\delta}_{T}$ be a solution to (\ref{Eq:EquationForReducedMLE}) and assuming $\theta\in(\theta_\ell,\theta_u)$, then
\[
\bar{\theta}^{\delta}_{T}=\tilde{\theta}^{\delta}_{T}\indicator{\{\tilde{\theta}^{\delta}_{T}\in(\theta_{\ell},\theta_{u})\}}+\theta_{\ell}\indicator{\tilde{\theta}^{\delta}_{T}\leq \theta_{\ell}\}}+\theta_{u}\indicator{\{\tilde{\theta}^{\delta}_{T}\geq \theta_{u}}.
\]
By Theorem \ref{T:ConsisitencyReducedLikelihood}, and based on smoothness of $\bar{h}_{\theta}$ as a function of $\theta$, asymptotic normality of the MLE corresponding to the reduced log-likelihood holds.

\begin{theorem}\label{T:CLTReducedLikelihood}
Assume Conditions \ref{A:Assumption1}, \ref{A:Assumption4} and that $\dot{\bar{h}}_{\theta}\doteq \frac{\partial \bar{h}_{\theta}}{\partial \theta}$ is continuous and for every $\theta\in\mathbb{R}^{d}$ the matrix $\dot{\bar{h}}_{\theta}^{*}\dot{\bar{h}}_{\theta}$ is positive definite. The maximum likelihood estimator based on (\ref{Eq:IntermediateLikelihood}) is asymptotically normal
under $\mathbb{P}_{\alpha}$, i.e.
\begin{equation}
\label{E:CLTReducedLikelihood}
\sqrt{T}\left(\bar{\theta}^{\delta}_{T}-\alpha\right)\Rightarrow N\left(0,\left(\dot{\bar{h}}^{*}_{\alpha}\dot{\bar{h}}_{\alpha}\right)^{-1}\right)\qquad\hbox{first as $\delta\downarrow 0$ and then $T\rightarrow\infty$.}
\end{equation}
\end{theorem}
\begin{proof} The proof is similar to that of Proposition 1.34 of \cite{Kutoyants}, even though there are no multiscale effects there. Below, we present the proof, emphasizing the differences due to the multiscale aspect of the present problem. 
Based on (\ref{Eq:EquationForReducedMLE}) for $\theta=\bar{\theta}^{\delta}_{T}$ we write
\[
\bar{h}_{\alpha}+\left(\bar{\theta}^{\delta}_{T}-\alpha\right)\dot{\bar{h}}_{\alpha^{*}}=\frac{1}{T}Y^{\delta}_{T}
\]
where $|\alpha^{*}-\alpha|\leq |\bar{\theta}^{\delta}_{T}-\alpha|$. Rearranging the latter expression we get
\begin{align}
\sqrt{T}\left(\bar{\theta}^{\delta}_{T}-\alpha\right)&=\sqrt{T}\left[\frac{1}{T}Y^{\delta}_{T}-\bar{h}_{\alpha}\right]\left(\dot{\bar{h}}_{\alpha^{*}}\right)^{-1} \nonumber\ .\
\end{align}
Now under the measure  $\mathbb{P}_{\alpha}$, we have that $Y^{\delta}_{T}=\int_{0}^{T}h_{\alpha}(X^{\delta}_{s})ds+W_{T}$. Hence, we can continue the latter expression as
\begin{align}
\sqrt{T}\left(\bar{\theta}^{\delta}_{T}-\alpha\right)
&=\sqrt{T}\left[\frac{1}{T}\int_{0}^{T}h_{\alpha}(X^{\delta}_{s})ds-\bar{h}_{\alpha}\right]\left(\dot{\bar{h}}_{\alpha^{*}}\right)^{-1} + \left[\frac{1}{\sqrt{T}}W_{T}\right]\left(\dot{\bar{h}}_{\alpha^{*}}\right)^{-1} \nonumber\\
&=\sqrt{T}\left[\frac{1}{T/\delta}\int_{0}^{T/\delta}h_{\alpha}(X^{1}_{s})ds-\bar{h}_{\alpha}\right]\left(\dot{\bar{h}}_{\alpha^{*}}\right)^{-1} + \left[\frac{1}{\sqrt{T}}W_{T}\right]\left(\dot{\bar{h}}_{\alpha^{*}}\right)^{-1} \nonumber\\
.\label{Eq:CLT_reduced1}
\end{align}
where we also used that $X^{\delta}_{\cdot}=X^{1}_{\cdot/\delta}$ in distribution. By taking $\delta\downarrow 0$ we have by the $L^{1}$ ergodic theorem  that
\[
\lim_{\delta\downarrow 0}\mathbb{E}\left|\frac{1}{T/\delta}\int_{0}^{T/\delta}h_{\alpha}(X^{1}_{s})ds-\bar{h}_{\alpha}\right|=0\qquad\hbox{for any $T\in(0,\infty)$.}
\]

Since $|\alpha^{*}-\alpha|\leq |\bar{\theta}^{\delta}_{T}-\alpha|$ we can apply the consistency of Theorem \ref{T:ConsisitencyReducedLikelihood} to get
\[\lim_{T\rightarrow\infty}\lim_{\delta\downarrow 0}\mathbb{P}_{\alpha}\left(\left|\alpha^*-\alpha\right|>\varepsilon\right)\leq \lim_{T\rightarrow\infty}\lim_{\delta\downarrow 0}\mathbb{P}_{\alpha}\left(\left|\bar\theta_T^\delta-\alpha\right|>\varepsilon\right)=0\qquad\hbox{for any $\varepsilon>0$,}\]
and hence by continuity we have $\dot{\bar{h}}_{\alpha^*}\rightarrow\dot{\bar{h}}_{\alpha}$ in probability as $\delta\downarrow 0$ and then $T\rightarrow\infty$.  Therefore,  by the positive definiteness of $\dot{\bar{h}}_{\alpha}$ we have the limit
\[
\sqrt{T}\left[\frac{1}{T/\delta}\int_{0}^{T/\delta}h_{\alpha}(X^{1}_{s})ds-\bar{h}_{\alpha}\right]\left(\dot{\bar{h}}_{\alpha^{*}}\right)^{-1} \rightarrow 0\qquad\hbox{in probability as $\delta\downarrow0$ and then $T\rightarrow\infty$.}
\]
For similar reasons, Slutsky's theorem implies
\[
\left[\frac{1}{\sqrt{T}}W_{T} \right]\left(\dot{\bar{h}}_{\alpha^{*}}\right)^{-1} \Rightarrow N\left(0,\left(\dot{\bar{h}}^{*}_{\alpha}\dot{\bar{h}}_{\alpha}\right)^{-1}\right)\qquad\hbox{first as $\delta\downarrow 0$ and then $T\rightarrow\infty$.}
\]
Finally, using Slutsky's theorem on the combined expression in (\ref{Eq:CLT_reduced1}) yields the statement of the theorem.
\end{proof}

\section{Simulation Example}\label{SS:simExample}
In this section, we present a simulation example, illustrating the theoretical findings. As an example, we consider the parameter space $\Theta\subset\mathbb R$, and take the true parameter value to be $\alpha=1$. We consider the model
\begin{align}
\nonumber
dX_t^\delta&=\frac 1\delta\left(\theta-X_t^\delta\right)dt+\sqrt\frac{2}{\delta}dB_t\\
dY_t^\delta&=\max(X_t^\delta,\theta)dt+dW_t\ ,
\label{E:simModel}
\end{align}
for $t\leq T= 5$ and $\delta = .01$. By Theorem 2.9 of \cite{KaratzasShreve}, there exists a unique strong solutions to the SDE for $(X^\delta,Y^\delta)$. For the purposes of the numerical example, we assume that the initial distribution of $X$ is its invariant law. If we run the system 2,000 times and each time compute $\bar \theta_t^\delta$, we get the histogram shown in Figure \ref{fig:barThetaHistFit}. For these trials, the MLE has empirical error of $.3180$, which is close to the $\frac{1}{\sqrt T}=.3162$ that is the standard error predicted by equation \eqref{E:CLTReducedLikelihood} in the CLT of Theorem \ref{T:CLTReducedLikelihood} with $\bar h_{\theta} = \theta+\frac{1}{2\sqrt\pi}$ and $\dot{\bar h}_{\theta} = 1$.

To show the effect of Theorem \ref{thm:CLT}, we compare the full log-likelihood to the reduced log-likelihood. The generator of the Ornstein-Uhlenbeck process in \eqref{E:simModel} has a discrete set of eigenvalues such that $\lambda_i^{\theta} = -i$ for $i=0,1,2,3,\dots$ for any $\theta\in\mathbb R$, and admits an orthonormal basis that is given (up to a normalizing constant) by the Hermite polynomials:
\[\psi_i^{\theta}(x) =\frac{1}{C_i}H_i(x-\theta) =  \frac{(-1)^i}{C_i}e^{\frac{(x-\theta)^2}{2}}\frac{d^i}{dx^i}e^{-\frac{(x-\theta)^2}{2}}\]
where $\psi$ is an eigenfunction as defined in Section \ref{SS:SpectralDecomposition}, and $H_i$ is the $ith$ (probabilist) Hermite polynomial (see \cite{Abramowitz}) and $C_i\doteq\sqrt{i!}$ is a normalizing constant. The eigen-coefficients of the function $h_{\theta}(x) = \max(x,\theta)$ are computed as follows:
\begin{align*}
\left<h_\theta,\psi_i^\theta\right>_\theta &= \frac{1}{\sqrt{2\pi}C_i}\int_{-\infty}^\infty \max(x,\theta)H_i(x-\theta)e^{-\frac{(x-\theta)^2}{2}}dx\\
&=\frac{1}{\sqrt{2\pi}C_i}\int_{-\infty}^\infty \left(\theta+\max(x-\theta,0)\right)H_i(x-\theta)e^{-\frac{(x-\theta)^2}{2}}dx\\
&=\frac{1}{\sqrt{2\pi}C_i}\int_{-\infty}^\infty \left(\theta+\max(u,0)\right)H_i(u)e^{-\frac{u^2}{2}}dx\\
&=\theta\cdot\indicator{i=0}+\frac{1}{\sqrt{2\pi}C_i}\int_0^\infty uH_i(u)e^{-\frac{u^2}{2}}dx\ ,
\end{align*}
and so only the zero order term depends on $\theta$ (the last computation used (\ref{Eq:OrthonormalBasis})). The eigen-coefficients are given in Table \ref{T:eigenCoefficients}. There is relatively fast decay among these coefficients, and hence, the limiting variance function $v^2(\theta)$ from Theorem \ref{thm:CLT} can be well-approximated by the first 15 to 20 basis elements.

The simulations and the analysis that follow demonstrate two things:
\begin{itemize}
 \item On one hand, $\rho^{\delta}_{t}(\theta)$ needs to be approximated based on methods such as Monte Carlo. As $\delta$ gets smaller  one needs more samples in order to compute $\rho^{\delta}_{t}(\theta)$ accurately .
\item On the other hand, the computation of $\bar{\rho}^{\delta}_{t}(\theta)$ is straightforward with no Monte Carlo errors. Theorem \ref{thm:CLT} quantifies the deviation of
$\bar{\rho}^{\delta}_{t}(\theta)$ from $\rho^{\delta}_{t}(\theta)$.
\end{itemize}

\begin{table}
\centering
\begin{tabular}{|c|cc cc cc c |c|}
\multicolumn{9}{c}{Eigen-Coefficients for $h_{\theta}(x) = \max(x,\theta)$.}\\
\hline
$\theta$&$i=0$&$i=1$&$i=2$&$i=3$&$i=4$&$i=5$&$i=6$&$v^2(h_\theta)$\\
\hline
.5&0.8989&    0.5000 &   0.2821  &       0&   -0.0814    &     0  &  0.0446& .04723\\
1&1.3989 &   0.5000 &   0.2821  &       0 &  -0.0814   &      0   & 0.0446 & .04723\\
1.5&1.8989&    0.5000&    0.2821  &       0&   -0.0814&         0 &   0.0446  & .04723\\
\hline
\end{tabular}
\vspace{.2cm}
\caption{The eigen-coefficients for $h_{\theta}(x) = \max(x,\theta)$ using the (normalized) Hermite polynomials. The limiting variance as predicted by Theorem \ref{thm:CLT} is given the last column, and is well-approximated by the first 15 to 20 basis elements. For this example, changes in $\theta$ only affect the first eigenmode.}
\label{T:eigenCoefficients}
\end{table}

To compute ${\rho}^{\delta}_{t}(\theta)$ we use Sequential Monte Carlo (SMC). Namely,
 we take independent samples $(X^{\delta,\ell})_{\ell=1}^N$ for some $N<\infty$ where each $X^{\delta,\ell}=_dX^\delta$, and our full log-likelihood is approximated as
\[\rho_t^\delta(\theta)\approx  \log\left(\frac 1N \sum_{\ell=1}^Ne^{-\frac 12\int_0^th_{\theta}^2(X_s^{\delta,\ell})ds+\int_0^th_{\theta}(X_s^{\delta,\ell})dY_s^\delta}\right)\ .\]
Estimation using SMC samples will have error that is of order $1/\sqrt N$, and with an asymptotically normal distribution (see \cite{delMoral2001,CMR2005})
\[\sqrt N\left(\log\left(\frac 1N \sum_{\ell=1}^Ne^{-\frac 12\int_0^th_{\theta}^2(X_s^{\delta,\ell})ds+\int_0^th_{\theta}(X_s^{\delta,\ell})dY_s^\delta}\right)-\rho_t^\delta(\theta) \right)\Rightarrow \mathcal Z(Y^\delta)\]
as $N\rightarrow \infty$, where $\mathcal Z(Y^\delta)$ is a normal random variable whose variance depends on the data $Y^\delta$.

In Figure \ref{fig:histFitsCLT} we see the histograms and fitted normal distributions obtained by looking at $\frac{1}{\sqrt t}\left(\rho_t^\delta(\theta)-\bar\rho_t^\delta(\theta)\right)$.
The solid red line is the density suggested by the CLT of Theorem \ref{thm:CLT}, namely a normal density with mean zero and variance $\delta v^{2}(\theta)$, and the dashed green line is a Gaussian density with mean zero and the empirical standard deviation. The Kolmogorov-Smirnoff test does not reject any of the empirical histogram fits to the green line (at the 99.9\%  confidence level), and the test rejects the histogram fits to the red lines for low confidence values and for different parameters. Heuristically, the difference in these standard errors should be $O(1/\sqrt N)$, \begin{align*}
\hbox{empirical standard error}&=\sqrt{var\left(\log\left(\frac 1N \sum_{\ell=1}^Ne^{-\frac 12\int_0^th_{\theta}^2(X_s^{\delta,\ell})ds+\int_0^th_{\theta}(X_s^{\delta,\ell})dY_s^\delta}\right)-\bar \rho_t^\delta(\theta)\right)}\\
&\leq \sqrt{var\left(\log\left(\frac 1N \sum_{\ell=1}^Ne^{-\frac 12\int_0^th_{\theta}^2(X_s^{\delta,\ell})ds+\int_0^th_{\theta}(X_s^{\delta,\ell})dY_s^\delta}\right)-\rho_t^\delta(\theta)\right)}\\
&\hspace{1cm}+\sqrt{var\left(\rho_t^\delta(\theta)-\bar\rho_t^\delta(\theta)\right)}\\
&\simeq O\left(\frac{1}{\sqrt N}\right)+\sqrt{\delta v^2(\theta)}\ .
\end{align*}
Indeed, from Table \ref{T:CLTvar} we see that the difference between the standard error of the CLT of Theorem \ref{thm:CLT} and the empirical standard error is of order $1/\sqrt N$, which indicates the strong possibility that the aforementioned error due to approximation via SMC is significant when estimating the log-likelihood.
\begin{table}
\centering
\begin{tabular}{|c|c|c|c|}
\multicolumn{4}{c}{Statistics for Simulations of  $\frac{1}{\sqrt t}\left(\rho_t^\delta(\theta)-\bar\rho_t^\delta(\theta)\right)$ with $\delta=.01$.}\\
\hline
$\theta$&$\sqrt{\delta v^2(\theta)}$&empirical std-err.&$\hbox{empirical std-err.}-\sqrt{\delta v^2(\theta)}$\\
\hline
.5&.02174&.0346&.0128\\
1&.02174&.0322&.0105\\
1.5&.02174&.0354&.0137\\
\hline
\end{tabular}
\vspace{.2cm}
\caption{For the model in \eqref{E:simModel}, 300 simulations of the quantity $\frac{1}{\sqrt t}\left(\rho_t^\delta(\theta)-\bar\rho_t^\delta(\theta)\right)$ computed with $N=2,000$, this table shows the standard error predicted by Theorem \ref{thm:CLT}, the empirical standard error, and the difference between the two. It turns out that $\frac{1}{\sqrt N} = \frac{1}{\sqrt{2,000}}=.0224$ which is of the same order as the entries in the 4th column, and so we conclude that the green line in Figure \ref{fig:histFitsCLT} has extra variance that is due to the SMC sampling error.}
\label{T:CLTvar}
\end{table}

Figure \ref{fig:histFitsCLT} indicates the following: not only is the reduced estimate of the log-likelihood close to the full likelihood,
but it might be a better estimate than a Monte Carlo approximation of the full log-likelihood. The enlarged Monte Carlo error in the computation of $\rho^{\delta}_{t}(\theta)$
can be seen in Figure \ref{fig:histFitsCLTlowDelta}, which is the same experiment, except with $\delta = .001$ (i.e. the same number of particles at $N=2,000$). In Figure \ref{fig:histFitsCLTlowDelta} it is important to notice how the Monte Carlo error is a significantly greater proportion of the total empirical error. If we want Figure \ref{fig:histFitsCLTlowDelta} to look similar to Figure \ref{fig:histFitsCLT}, then we would need to increase  $N$ by a factor of 10. Such an increase in the number of particles would significantly increase the computation time. Hence, the reduced filter outperforms the direct Monte Carlo filter for $\delta\ll1$, which is a motivation for this paper.

\begin{figure}[htbp] 
   \centering
   \includegraphics[height=3in, width=6in]{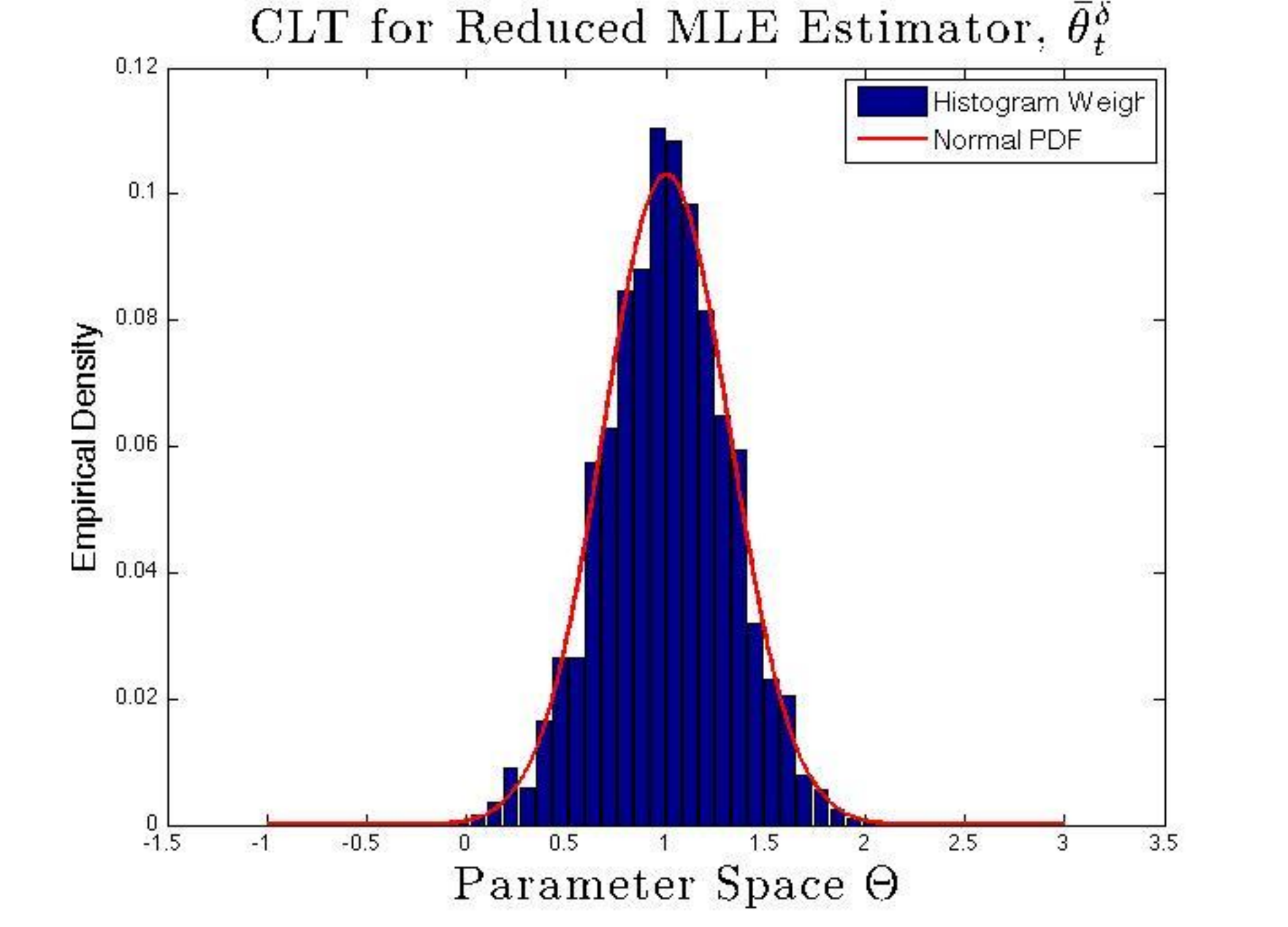}
  \caption{The empirical distribution of the reduced estimator $\bar\theta_t^\delta$, for which the asymptotic distribution is close to Gaussian. We run the system 2,000 times and each time compute $\bar \theta_t^\delta$. For these trials, the MLE has empirical error of $.3180$, which is close to the $\frac{1}{\sqrt T}=.3162$ that is the standard error predicted by equation \eqref{E:CLTReducedLikelihood} in the CLT of Theorem \ref{T:CLTReducedLikelihood} with $\bar h_{\theta} = \theta+\frac{1}{2\sqrt\pi}$ and $\dot{\bar h}_{\theta} = 1$}
   \label{fig:barThetaHistFit}
\end{figure}

\begin{figure}[htbp] 
   \centering
   \includegraphics[height=3in,width=6.5in]{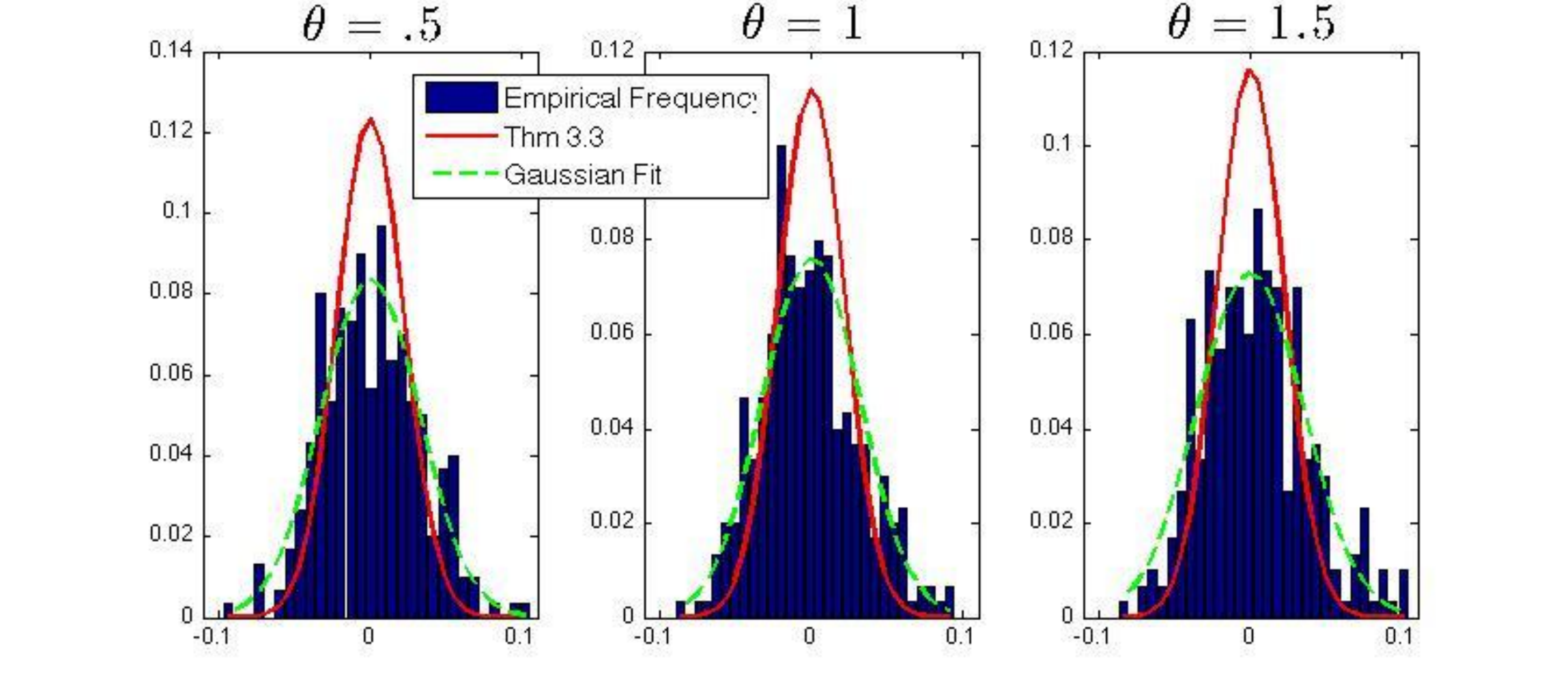}
  \caption{Histograms of the quantity $\frac{1}{\sqrt t}\left(\rho_t^\delta(\theta)-\bar\rho_t^\delta(\theta)\right)$ for $\theta=.5,1,1.5$ with the true parameter being $\alpha=1$. The solid red line is the limiting Gaussian distribution of Theorem \ref{thm:CLT}, and the dashed green line is a Gaussian fit to the histogram. The green line has a slightly greater standard deviation because $\rho_t^\delta(\theta)$ needs to be approximated with Monte Carlo sampling and a discrete time scheme, and hence, the empirical distribution has some additional variance. However, the Kolmogorov-Smirnoff test does not reject the hypothesis that the histogram is a Gaussian distribution.}
   \label{fig:histFitsCLT}
\end{figure}

\begin{figure}[htbp] 
   \centering
   \includegraphics[height=3in,width=6.5in]{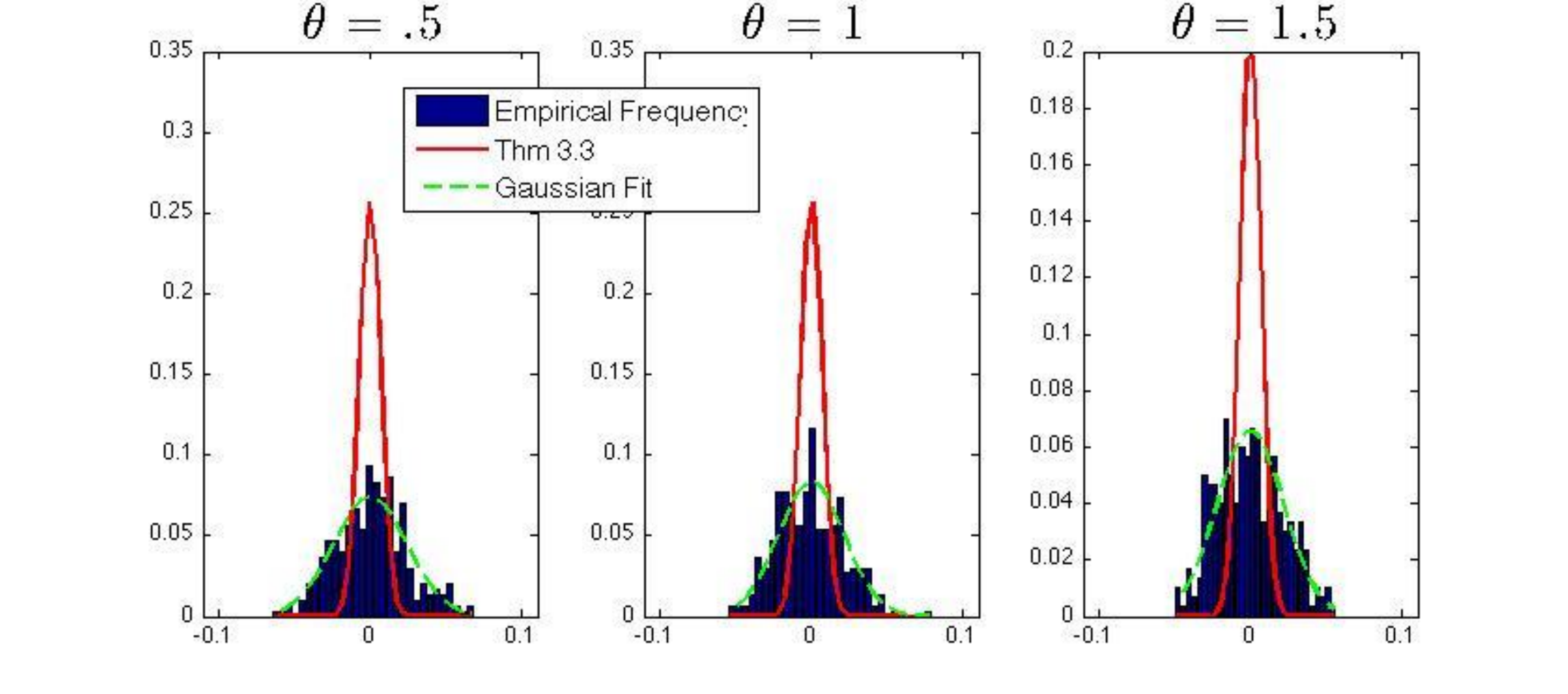}
  \caption{The same histograms as those in Figure \ref{fig:histFitsCLT}, except for the experiment run with $\delta=.001$. Notice how the Monte Carlo error is a greater proportion of the total variance. This illustrates how for $\delta $ small, the reduced likelihood can be more accurate than a Monte Carlo approximation.}
   \label{fig:histFitsCLTlowDelta}
\end{figure}

\section{Conclusions \& Future Work}\label{S:Conclusions}
This paper studies parameter estimation with partially observed diffusions of models with multiple time scales. This problem is primarily an application of ergodic theory to nonlinear filtering. We prove convergence in probability of the nonlinear filter and of the conditional (on the observations) log-likelihood. Furthermore, we prove a central limit theorem for the log-likelihood. These results justify the use of a log-likelihood of reduced dimension for the purposes of parameter estimation, which is simpler to implement and has faster runtime in computations. Consistency and asymptotic normality for the MLE of the reduced log-likelihood is also obtained, and simulation studies are presented to show how the reduced log-likelihood can outperform a direct Monte Carlo filter when $\delta \ll1$.

It is plausible that some of the results presented in this paper can be generalized. For instance, it is possible that the CLT can be proven with the removal of the assumption of $h_\theta$ being bounded, which is also supported by the simulation example of Section \ref{SS:simExample}. Regarding the generalization of Theorem \ref{thm:CLT}, it may be possible to prove a version of the theorem using generalized spectral theory rather than assuming a discrete spectrum with orthonormal eigenfunctions, but modifications to the techniques developed in this paper will be needed.

\appendix

\section{Proof of Theorem \ref{T:FilterConvergence1}}
\label{A:FilterConvergence1}

The proof of Theorem \ref{T:FilterConvergence1} follows by the results of \cite{ImkellerSriPerkowskiYeong2012}, see also \cite{ParkSriSowers2008, ParkRozovskySowers2010,ParkSriSowers2011,ImkellerSriPerkowskiYeong2012} after we adjust for the parameter mismatch. In particular, the main difference that Theorem \ref{T:FilterConvergence1} has when compared to the previous works is that
under the measure parameterized  by the true parameter value (i.e. the measure under which the observations are made) the filters will converge for \textit{any} parameter value.  Moreover, we also need to prove that the convergence of the filters is for test functions in the space the space $\mathcal{A}_{\eta}^{\theta}$, whereas the results in \cite{ImkellerSriPerkowskiYeong2012} use bounded and smooth test functions.

\begin{lemma}
 \label{L:FilterConvergence4}
Let us consider $f\in C^{4}_{b}(\mathcal X)$ and assume  Conditions \ref{A:Assumption1} and \ref{A:Assumption2}. For any $\theta,\alpha\in\Theta$, we have uniformly in $t\in[0,T]$
\[
\mathbb E_{\alpha}\left|\phi_t^{\delta,\theta}[f]-\bar{\phi}_t^{\delta,\theta}[f]\right|\rightarrow 0\qquad\hbox{as }\delta\rightarrow0\ .
\]

\end{lemma}
\begin{proof}
By H\"{o}lder inequality, for $p,q<0$ with $\frac 1p+\frac 1q=1$ we have
\begin{align}
\mathbb E_{\alpha}\left|\phi_t^{\delta,\theta}[f]-\bar{\phi}_t^{\delta,\theta}[f]\right|&=
\mathbb E^{*}_{\alpha}\left[Z^{\delta,\alpha}_{t}\left|\phi_t^{\delta,\theta}[f]-\bar{\phi}_t^{\delta,\theta}[f]\right|\right]\nonumber\\
&\leq
 \left(\mathbb E^{*}_{\alpha}\left|Z^{\delta,\alpha}_{t}\right|^{q}\right)^{1/q}\left(\mathbb E^{*}_{\alpha}\left|\phi_t^{\delta,\theta}[f]-\bar{\phi}_t^{\delta,\theta}[f]\right|^{p}\right)^{1/p}\nonumber\\
&=
 \left(\mathbb E^{*}_{\alpha}\left|Z^{\delta,\alpha}_{t}\right|^{q}\right)^{1/q}\left(\mathbb E^{*}_{\theta}\left|\phi_t^{\delta,\theta}[f]-\bar{\phi}_t^{\delta,\theta}[f]\right|^{p}\right)^{1/p}\nonumber\\
 &\leq
 \left(\mathbb E^{*}_{\alpha}\left|Z^{\delta,\alpha}_{t}\right|^{q}\right)^{1/p} \left(\mathbb E_{\theta}\left|Z^{\delta,\theta}_{t}\right|^{-q}\right)^{1/(pq)}\left(\mathbb E_{\theta}\left|\phi_t^{\delta,\theta}[f]-\bar{\phi}_t^{\delta,\theta}[f]\right|^{p^{2}}\right)^{1/p^{2}}\nonumber
\end{align}
which goes to zero as $\delta\downarrow 0$ by Condition \ref{A:Assumption2} and Lemma 6.6 in \cite{ImkellerSriPerkowskiYeong2012}. The third line,  i.e., $\mathbb E_{\alpha}^*\left|\phi_t^{\delta,\theta}[f]-\bar{\phi}_t^{\delta,\theta}[f]\right|^{p}=\mathbb E_{\theta}^*\left|\phi_t^{\delta,\theta}[f]-\bar{\phi}_t^{\delta,\theta}[f]\right|^{p}$, follows because both $\phi_t^{\delta,\theta}$ and $\bar\phi_t^{\delta,\theta}$ are functionals of $Y^\delta_{\cdot}$ (and no other random variable), and $Y^\delta$ is a Brownian motion under both measures $\mathbb P_{\alpha}^*$ and $\mathbb{P}_{\theta}^*$. This concludes the proof of the lemma.
\end{proof}


We conclude with the proof of Theorem \ref{T:FilterConvergence1}.

\begin{proof}[Proof of Theorem \ref{T:FilterConvergence1}]
Lemma \ref{L:FilterConvergence4} implies convergence in probability:
\[\mathbb P_{\alpha}\left(\left|\phi_t^{\delta,\theta}[f]-\bar\phi_t^{\delta,\theta}[f]\right|>\eps\right)\leq \frac{1}{\eps}\mathbb E_{\alpha}\left|\phi_t^{\delta,\theta}[f]-\bar\phi_t^{\delta,\theta}[f]\right|\rightarrow 0\qquad\forall\eps>0\]
for bounded $f$. Let us now prove the second part of the theorem. We prove it first for $f\in C^{4}_{b}(\mathcal{X})$. Then, we prove it under the assumption that there exists $\eta>0$ such that $f\in\mathcal{A}_{\eta}^{\theta}$. So, let us assume that $f\in C^{4}_{b}(\mathcal{X})$. It is clear that by ergodicity we have
\[\lim_{\delta\downarrow 0}\left|\bar \pi_t^{\delta,\theta}[f]-\bar \pi_t^{\theta}[f]\right|=0\hspace{.5cm}\hbox{in $\mathbb P_\alpha$ probability}\ ,\]
so it remains to prove that
\[\lim_{\delta\downarrow 0}\mathbb E_{\alpha}\left(\pi_t^{\delta,\theta}[f]-\bar \pi_t^{\delta, \theta}[f]\right)^2=0\ . \]
For this purpose, H\"{o}lder inequality gives

\begin{align}
\mathbb E_{\alpha}\left|\pi_t^{\delta,\theta}[f]-\bar{\pi}_t^{\delta,\theta}[f]\right|^{2}&=
\mathbb E^{*}_{\alpha}\left[Z^{\delta,\alpha}_{t}\left|\pi_t^{\delta,\theta}[f]-\bar{\pi}_t^{\delta,\theta}[f]\right|^{2}\right]\nonumber\\
&\leq
 \left(\mathbb E^{*}_{\alpha}\left|Z^{\delta,\alpha}_{t}\right|^{q}\right)^{1/q}\left(\mathbb E^{*}_{\alpha}\left|\pi_t^{\delta,\theta}[f]-\bar{\pi}_t^{\delta,\theta}[f]\right|^{2p}\right)^{1/p}\nonumber\\
&\leq
 \left(\mathbb E^{*}_{\alpha}\left|Z^{\delta,\alpha}_{t}\right|^{q}\right)^{1/q}\left(\mathbb E^{*}_{\theta}\left|\pi_t^{\delta,\theta}[f]-\bar{\pi}_t^{\delta,\theta}[f]\right|^{2p}\right)^{1/p}\nonumber\\
 &\leq
 \left(\mathbb E^{*}_{\alpha}\left|Z^{\delta,\alpha}_{t}\right|^{q}\right)^{1/q} \left(\mathbb E_{\theta}\left|Z^{\delta,\theta}_{t}\right|^{-q}\right)^{1/(pq)}\left(\mathbb E_{\theta}\left|\pi_t^{\delta,\theta}[f]-\bar{\pi}_t^{\delta,\theta}[f]\right|^{2p^{2}}\right)^{1/p^{2}}\nonumber
\end{align}
which goes to zero as $\delta\downarrow 0$ by Condition \ref{A:Assumption2} and Corollary 6.9 in \cite{ImkellerSriPerkowskiYeong2012}. The third line,  i.e., $\mathbb E_{\alpha}^*\left|\phi_t^{\delta,\theta}[f]-\bar{\phi}_t^{\delta,\theta}[f]\right|^{2p}=\mathbb E_{\theta}^*\left|\phi_t^{\delta,\theta}[f]-\bar{\phi}_t^{\delta,\theta}[f]\right|^{2p}$, follows because both $\phi_t^{\delta,\theta}$ and $\bar\phi_t^{\delta,\theta}$ are functionals of $Y^\delta_{\cdot}$ (and no other random variable), and $Y^\delta$ is a Brownian motion under both measures $\mathbb P_{\alpha}^*$ and $\mathbb{P}_{\theta}^*$. This completes the proof for $f\in C^{4}_{b}(\mathcal{X})$.

Let us complete the proof of the theorem by assuming that there exists an $\eta>0$ such that $f\in\mathcal{A}_{\eta}^\theta$. For $n\in\mathbb{N}$,   define
\[
u_{n}(x)=
\begin{cases}
x &,|x|\leq n\\
n \textrm{ sign}(x)&,|x|>n
\end{cases}
\]
and set $f_{n}(x)=u_{n}(f(x))$. Analogously define
\[
\pi^{\delta,\theta}_{t}[f_{n}]\doteq\mathbb E_{\theta}\left[f_{n}(X_t^\delta)\Big|\mathcal Y_t^\delta\right],\quad \bar{f}_{n,\theta}=\int_{\mathcal{X}}f_{n}(x)\mu_{\theta}(dx).
\]
Since $f_{n}$ is bounded, we already know that $\lim_{\delta\downarrow 0}\mathbb E_{\alpha}\left|\pi_t^{\delta,\theta}[f_{n}]-\bar \pi_t^{\theta}[f_{n}]\right|^2=0$.
So, it is enough to prove that
\[
\lim_{n\rightarrow\infty}\limsup_{\delta\downarrow 0}\mathbb E_{\alpha}\left|\pi_t^{\delta,\theta}[f]-\pi_t^{\delta,\theta}[f_{n}]\right|^2=0
\]
and
\[
\lim_{n\rightarrow\infty}\left|\bar \pi_{t}^{\theta}[f]-\bar \pi_t^{\theta}[f_{n}]\right|^2=0\ .
\]
Both of these statements follow from the observation
\[
|f(x)-f_{n}(x)|^{2p^{2}}\leq |f(x)|^{2p^{2}}\indicator{|f(x)|>n}\leq |f(x)|^{2+\eta}\indicator{|f(x)|>n}\leq n^{-\eta}|f(x)|^{2+\eta}
\]
In particular, we have
\begin{align}
&\lim_{n\rightarrow\infty}\limsup_{\delta\downarrow 0}\mathbb E_{\alpha}\left(\pi_t^{\delta,\theta}[f]-\pi_t^{\delta,\theta}[f_{n}]\right)^2=
\lim_{n\rightarrow\infty}\limsup_{\delta\downarrow 0}\mathbb E_{\alpha}\left(\mathbb{E}_{\theta}\left[f(X^{\delta}_{t})-f_{n}(X^{\delta}_{t})\Big| \mathcal{Y}^{\delta}_{t}\right]\right)^2\nonumber\\
&\qquad\leq
\lim_{n\rightarrow\infty}\limsup_{\delta\downarrow 0}\mathbb E_{\alpha}\mathbb{E}_{\theta}\left[\left|f(X^{\delta}_{t})-f_{n}(X^{\delta}_{t})\right|^2\Big| \mathcal{Y}^{\delta}_{t}\right]\nonumber\\
&\qquad\leq
\lim_{n\rightarrow\infty}\limsup_{\delta\downarrow 0}\left(\mathbb E_{\alpha}^*\left|Z_t^{\delta,\alpha}\right|^{q}\right)^{1/q}\left(\mathbb E_{\theta}\left|Z_t^{\delta,\theta}\right|^{-q}\right)^{1/(pq)}\left(\mathbb E_{\theta}\left|f(X_t^\delta)- f_n(X_t^\delta)\right|^{2p^{2}}\right)^{1/p^{2}}\nonumber\\
&\qquad\leq
2\lim_{n\rightarrow\infty}\limsup_{\delta\downarrow 0}n^{-\eta/p^{2}}\left(\mathbb E_{\alpha}^*\left|Z_t^{\delta,\alpha}\right|^{q}\right)^{1/q}\left(\mathbb E_{\theta}\left|Z_t^{\delta,\theta}\right|^{-q}\right)^{1/(pq)}\left( \mathbb  E_{\theta}\left|f(X^{\delta}_{t})\right|^{2+\eta}\right)^{1/p^{2}}\nonumber\\
&\qquad=0
\end{align}
and clearly $\lim_{n\rightarrow\infty}\left(\bar \pi_{t}^{\theta}[f]-\bar \pi_t^{\theta}[f_{n}]\right)^2=\lim_{n\rightarrow\infty}\left(\bar f_{\theta}-\bar f_{n,\theta}\right)^2=0$. This concludes the proof of the theorem.
\end{proof}

\section{Some Convergence Results for the Posterior Expectation of the Eigenfunctions}\label{A:lemmasAndProps}
In this subsection, we collect a number of results associated with the asymtpotic behavior of the posterior for the eigenfunctions and their correlation as $\delta\downarrow 0$. Recall that $\alpha$ denotes the true parameter value.

\begin{lemma}
\label{L:parameterMismatch}
 Suppose $h_\theta$ is uniformly bounded over $\theta$ by a constant $C_h<\infty$ such that $\sup_{\theta\in\Theta}\|h_\theta\|_\infty\leq C_h$. Then there exists another constant $C_\mu<\infty$ such that
\[\sup_{\alpha,\theta\in\Theta}\mathbb E_\theta^*\left[\left(\frac{\phi_t^{\delta,\alpha}[1]}{\phi_t^{\delta,\theta}[1]}\right)^2\right]\leq C_\mu\ ,\]
and for any $f\in \mathcal A_\theta^\eta$ with $f\geq 0$ we have
\[\mathbb E_\alpha \pi_t^{\delta,\theta}[f]\leq  \frac{\mathbb E_\theta\left[f(X_t^\delta)\right]}{2}\left(e^{tC_h^2 }+C_\mu\right)\]
for any $\theta,\alpha\in\Theta$ and for any $t\in[0,T]$.
\end{lemma}
\begin{proof} From the Cauchy inequality (i.e. $ab\leq a^{2}/2+b^{2}/2$ for all $a,b\in\mathbb R$), we have the following uniform bound:
\begin{align*}
\mathbb E_\theta^*\left[\left(\frac{\phi_t^{\delta,\alpha}[1]}{\phi_t^{\delta,\theta}[1]}\right)^2\right]&\leq \frac 12\mathbb E_\theta^*\left[\left(\phi_t^{\delta,\alpha}[1]\right)^4\right]+\frac 12\mathbb E_\theta^*\left[\left(\frac{1}{\phi_t^{\delta,\theta}[1]}\right)^4\right]\\
&\leq \frac 12\mathbb E_\theta^*\left[\left(Z_t^{\delta,\alpha}\right)^4\right]+\frac 12\mathbb E_\theta^*\left[\left(Z_t^{\delta,\theta}\right)^{-4}\right]\\
&= \frac 12\mathbb E_\theta^*\mathbb E_\theta^*\left[\left(Z_t^{\delta,\alpha}\right)^4\Big|(X_s^\delta)_{s\leq t}\right]+\frac 12\mathbb E_\theta^*\mathbb E_\theta^*\left[\left(Z_t^{\delta,\theta}\right)^{-4}\Big|(X_s^\delta)_{s\leq t}\right]\\
&= \frac 12\mathbb E_\theta^*\left[e^{6\int_0^t|h_\alpha(X_s^\delta)|^2ds}\right]+\frac{1}{2}\mathbb E_\theta^*\left[e^{12\int_0^t|h_\theta(X_s^\delta)|^2ds}\right]\\
&\leq \frac{e^{6TC_h^2}+e^{12 TC_h^2}}{2}\\
&<\infty\ .
\end{align*}
This proves the first statement of the lemma with the constant being $C_\mu \doteq\frac{1}{2}\left(e^{6TC_h^2}+e^{12TC_h^2}\right)$. To prove the lemma's second statement, we take any $f\in \mathcal A_\theta^\eta$  with $f\geq 0$, and proceed as follows:
\begin{align*}
\mathbb E_\alpha\pi_t^{\delta,\theta}[f]
&=\frac{\mathbb E_\alpha^*\left[Z_t^{\delta,\alpha}\pi_t^{\delta,\theta}[f]\right]}{\mathbb E_\alpha^*\left[Z_t^{\delta,\alpha}\right]}\\
&=\mathbb E_\alpha^*\left[Z_t^{\delta,\alpha}\pi_t^{\delta,\theta}[f]\right]=\mathbb E_\alpha^*\left[\mathbb E_\alpha^*\left[Z_t^{\delta,\alpha}\pi_t^{\delta,\theta}[f]\Big|\mathcal Y_t^\delta\right]\right]\\
&=\mathbb E_\alpha^*\left[\mathbb E_\alpha^*\left[Z_t^{\delta,\alpha}\Big|\mathcal Y_t^\delta\right]\pi_t^{\delta,\theta}[f]\right]=\mathbb E_\alpha^*\left[\phi_t^{\delta,\alpha}[1]\pi_t^{\delta,\theta}[f]\right]
\end{align*}
and because $(Y_t^\delta)_{t\leq T}$ is  Brownian motion under both $\mathbb P_\alpha^*$ and $\mathbb P_\theta^*$, we have that the last display continues as
\begin{align*}
=\mathbb E_\theta^*\left[\phi_t^{\delta,\alpha}[1]\pi_t^{\delta,\theta}[f]\right]&=
\mathbb E_\theta^*\left[\frac{\phi_t^{\delta,\alpha}[1]}{\phi_t^{\delta,\theta}[1]}\phi_t^{\delta,\theta}[f]\right]\\
&=\mathbb E_\theta^*\left[\frac{\phi_t^{\delta,\alpha}[1]}{\phi_t^{\delta,\theta}[1]}\mathbb E_\theta^*\left[Z_t^{\delta,\theta}f(X_t^\delta)\Big|\mathcal Y_t^\delta\right]\right]\\
&=\mathbb E_\theta^*\left[\mathbb E_\theta^*\left[\frac{\phi_t^{\delta,\alpha}[1]}{\phi_t^{\delta,\theta}[1]}Z_t^{\delta,\theta}f(X_t^\delta)\Big|\mathcal Y_t^\delta\right]\right]\\
&=\mathbb E_\theta^*\left[\frac{\phi_t^{\delta,\alpha}[1]}{\phi_t^{\delta,\theta}[1]}Z_t^{\delta,\theta}f(X_t^\delta)\right]\\
&=\mathbb E_\theta^*\left[f(X_t^\delta)\mathbb E_\theta^*\left[\frac{\phi_t^{\delta,\alpha}[1]}{\phi_t^{\delta,\theta}[1]}Z_t^{\delta,\theta}\Big|( X_s^\delta)_{s\leq t}\right]\right]\\
&\leq\frac 12\mathbb E_\theta^*\left[f(X_t^\delta)\mathbb E_\theta^*\left[\left(Z_t^{\delta,\theta}\right)^2+\left(\frac{\phi_t^{\delta,\alpha}[1]}{\phi_t^{\delta,\theta}[1]}\right)^2\Big|(X_s^\delta)_{s\leq t}\right]\right]\\
&=\frac 12\mathbb E_\theta^*\left[f(X_t^\delta)\exp\left(\int_0^t|h_\theta^\delta(X_s^\delta)|^2ds\right)\right]+\frac 12\mathbb E_\theta^*\left[f(X_t^\delta)\left(\frac{\phi_t^{\delta,\alpha}[1]}{\phi_t^{\delta,\theta}[1]}\right)^2\right]\\
&=\frac 12\mathbb E_\theta^*\left[f(X_t^\delta)\exp\left(\int_0^t|h_\theta^\delta(X_s^\delta)|^2ds\right)\right]+\frac 12\mathbb E_\theta^*\left[f(X_t^\delta)\right]\underbrace{\mathbb E_\theta^*\left[\left(\frac{\phi_t^{\delta,\alpha}[1]}{\phi_t^{\delta,\theta}[1]}\right)^2\right]}_{\leq C_\mu}\\
&\leq  \frac{\mathbb E_\theta^*\left[f(X_t^\delta)\right]}{2}\left(e^{tC_h^2 }+C_\mu\right)\ .
\end{align*}
This concludes the proof of the lemma.
\end{proof}

\begin{lemma}\label{L:UsedForTightnessPsi}
Assume Conditions \ref{A:Assumption1} and \ref{A:Assumption3}. For any $s\in(0,T]$, for any $\alpha,\theta\in\Theta$, we have that there exists a constant $C_{i,j}(T,\theta)<\infty$ that may depend on $i,j,T,\theta$ but does not depend on $\delta$ such that for $\delta$ small enough
\begin{align*}
&\sup_{i,j\geq 1}\frac{1}{\delta}\left|\mathbb{E}_{\alpha}\pi_s^{\delta,\theta}[\psi_{i}^{\theta}]\pi_s^{\delta,\theta}[\psi_{j}^{\theta}]\right|
\leq C_{i,j}(T,\theta)
\end{align*}
In addition, we also have that
\[
\lim_{\delta\downarrow 0}\frac{1}{\delta}\mathbb{E}_{\alpha}\left[\pi^{\delta,\theta}_{s}[\psi_{i}^{\theta}]\pi^{\delta,\theta}_{s}[\psi_{j}^{\theta}]\right]=\frac{\left<h_{\theta},\psi_i^{\theta}\right>_{\theta} \left<h_{\theta},\psi_j^{\theta}\right>_{\theta}}{\lambda_i^{\theta}+\lambda_j^{\theta}}.
\]
\end{lemma}
\begin{proof}
Based on (\ref{eq:KSsoln2}), we can write
\begin{align}
&\frac{1}{\delta}\mathbb{E}_{\alpha}\left[\pi_s^{\delta,\theta}[\psi_{i}^{\theta}]\pi_s^{\delta,\theta}[\psi_{j}^{\theta}]\right]\nonumber\\
&=\frac{1}{\delta}e^{-\frac{\lambda_i^{\theta}+\lambda_j^{\theta}}{\delta}s}\pi^{\theta}_{0}[\psi_{i}^{\theta}]\pi^{\theta}_{0}[\psi_{j}^{\theta}]\nonumber\\
&\hspace{0.5cm}+\frac{1}{\delta}\int_{0}^{s}
\mathbb{E}_{\alpha}\left\{\left[e^{-\frac{\lambda_i^{\theta}}{\delta}(s-\rho)}\pi^{\delta,\theta}_{\rho}[\psi_{j}^{\theta}]\left(\pi^{\delta,\theta}_{\rho}[h_{\theta}\psi_{i}^{\theta}]-\pi^{\delta,\theta}_{\rho}[\psi_{i}^{\theta}]\pi^{\delta,\theta}_{\rho}[h_{\theta}]   \right)+\right.\right.\nonumber\\
&\hspace{1cm}\left.\left.+e^{-\frac{\lambda_j^{\theta}}{\delta}(s-\rho)}\pi^{\delta,\theta}_{\rho}[\psi_{i}^{\theta}]\left(\pi^{\delta,\theta}_{\rho}[h_{\theta}\psi_{j}^{\theta}]-\pi^{\delta,\theta}_{\rho}[\psi_{j}^{\theta}]\pi^{\delta,\theta}_{\rho}[h_{\theta}]   \right)\right]\left(\pi^{\delta,\theta}_{\rho}[h_{\theta}]-\pi^{\delta,\alpha}_{\rho}[h_{\alpha}]\right)\right\}d\rho\nonumber\\
&\hspace{0.5cm}+\frac{1}{\delta}\int_{0}^{s}e^{-\frac{\lambda_i^{\theta}+\lambda_j^{\theta}}{\delta}(s-\rho)}
\mathbb{E}_{\alpha}\left[\left(\pi^{\delta,\theta}_{\rho}[h_{\theta}\psi_{i}^{\theta}]-\pi^{\delta,\theta}_{\rho}[\psi_{i}^{\theta}]\pi^{\delta,\theta}_{\rho}[h_{\theta}]   \right)\left(\pi^{\delta,\theta}_{\rho}[h_{\theta}\psi_{j}^{\theta}]-\pi^{\delta,\theta}_{\rho}[\psi_{j}^{\theta}]\pi^{\delta,\theta}_{\rho}[h_{\theta}]   \right)\right]d\rho \ ,
\label{Eq:UsedForBoundPsi}
\end{align}
and then taking absolute values inside the integrals, applying the Cauchy inequality ($ab\leq a^{2}/2+b^{2}/2$ for all $a,b\in\mathbb R$) and applying Lemma \ref{L:parameterMismatch}, we have the following bound:
\begin{align}
&\frac{1}{\delta}\mathbb{E}_{\alpha}\left|\pi_s^{\delta,\theta}[\psi_{i}^{\theta}]\pi_s^{\delta,\theta}[\psi_{j}^{\theta}]\right|\nonumber\\
&\leq \frac{1}{\delta}e^{-\frac{\lambda_i^{\theta}+\lambda_j^{\theta}}{\delta}s}\left|\pi^{\theta}_{0}[\psi_{i}^{\theta}]\pi^{\theta}_{0}[\psi_{j}^{\theta}]\right|\nonumber\\
&\hspace{0.5cm}+\frac{1}{\delta}\int_{0}^{s}
\mathbb{E}_{\alpha}\left[\left(e^{-\frac{\lambda_i^{\theta}}{\delta}(s-\rho)}\left|\pi^{\delta,\theta}_{\rho}[\psi_{j}^{\theta}]\right|\left|\pi^{\delta,\theta}_{\rho}[h_{\theta}\psi_{i}^{\theta}]-\pi^{\delta,\theta}_{\rho}[\psi_{i}^{\theta}]\pi^{\delta,\theta}_{\rho}[h_{\theta}]   \right|+\right.\right.\nonumber\\
&\hspace{1cm}\left.\left.+e^{-\frac{\lambda_j^{\theta}}{\delta}(s-\rho)}\left|\pi^{\delta,\theta}_{\rho}[\psi_{i}^{\theta}]\right|\left|\pi^{\delta,\theta}_{\rho}[h_{\theta}\psi_{j}^{\theta}]-\pi^{\delta,\theta}_{\rho}[\psi_{j}^{\theta}]\pi^{\delta,\theta}_{\rho}[h_{\theta}]   \right|\right)\left|\pi^{\delta,\theta}_{\rho}[h_{\theta}]-\pi^{\delta,\alpha}_{\rho}[h_{\alpha}]\right|\right]d\rho\nonumber\\
&\hspace{0.5cm}+\frac{1}{\delta}\int_{0}^{s}e^{-\frac{\lambda_i^{\theta}+\lambda_j^{\theta}}{\delta}(s-\rho)}
\mathbb{E}_{\alpha}\left[\left|\pi^{\delta,\theta}_{\rho}[h_{\theta}\psi_{i}^{\theta}]-\pi^{\delta,\theta}_{\rho}[\psi_{i}^{\theta}]\pi^{\delta,\theta}_{\rho}[h_{\theta}]   \right|\left|\pi^{\delta,\theta}_{\rho}[h_{\theta}\psi_{j}^{\theta}]-\pi^{\delta,\theta}_{\rho}[\psi_{j}^{\theta}]\pi^{\delta,\theta}_{\rho}[h_{\theta}]   \right|\right]d\rho \nonumber\\
&\leq \frac{1}{\delta}e^{-\frac{\lambda_i^{\theta}+\lambda_j^{\theta}}{\delta}s}\left|\pi^{\theta}_{0}[\psi_{i}^{\theta}]\pi^{\theta}_{0}[\psi_{j}^{\theta}]\right|\nonumber\\
&\hspace{.5cm}+\frac{1}{2\delta}\int_{0}^{s}
\mathbb{E}_{\alpha}\left[e^{-\frac{\lambda_i^{\theta}}{\delta}(s-\rho)}\left[\left|\pi^{\delta,\theta}_{\rho}[\psi_{j}^{\theta}]\right|^2+\left|\pi^{\delta,\theta}_{\rho}[h_{\theta}\psi_{i}^{\theta}]-\pi^{\delta,\theta}_{\rho}[\psi_{i}^{\theta}]\pi^{\delta,\theta}_{\rho}[h_{\theta}]   \right|^2+\right.\right.\nonumber\\
&\hspace{1cm}\left.\left.+e^{-\frac{\lambda_j^{\theta}}{\delta}(s-\rho)}\left|\pi^{\delta,\theta}_{\rho}[\psi_{i}^{\theta}]\right|^2+\left|\pi^{\delta,\theta}_{\rho}[h_{\theta}\psi_{j}^{\theta}]-\pi^{\delta,\theta}_{\rho}[\psi_{j}^{\theta}]\pi^{\delta,\theta}_{\rho}[h_{\theta}]   \right|^2\right]2C_h\right]d\rho\nonumber\\
&\hspace{0.5cm}+\frac{1}{2\delta}\int_{0}^{s}e^{-\frac{\lambda_i^{\theta}+\lambda_j^{\theta}}{\delta}(s-\rho)}
\mathbb{E}_{\alpha}\left[\left|\pi^{\delta,\theta}_{\rho}[h_{\theta}\psi_{i}^{\theta}]-\pi^{\delta,\theta}_{\rho}[\psi_{i}^{\theta}]\pi^{\delta,\theta}_{\rho}[h_{\theta}]   \right|^2+\left|\pi^{\delta,\theta}_{\rho}[h_{\theta}\psi_{j}^{\theta}]-\pi^{\delta,\theta}_{\rho}[\psi_{j}^{\theta}]\pi^{\delta,\theta}_{\rho}[h_{\theta}]   \right|^2\right]d\rho \nonumber\\
&\leq C_{0}\left(\frac{1}{\lambda_i^{\theta}+\lambda_j^{\theta}}\frac{\lambda_i^{\theta}+\lambda_j^{\theta}}{\delta}e^{-\frac{\lambda_i^{\theta}+\lambda_j^{\theta}}{\delta}s}\left|\pi^{\theta}_{0}[\psi_{i}^{\theta}]\pi^{\theta}_{0}[\psi_{j}^{\theta}]\right|\right.\label{Eq:BoundPsi}\\
&\left.
+\frac 1\delta\int_0^s\left[e^{-\frac{\lambda_i^{\theta}}{\delta}(s-\rho)}+
e^{-\frac{\lambda_j^{\theta}}{\delta}(s-\rho)}\right]\mathbb E_\theta\left[|\psi_i^\theta(X_\rho^\delta)|^2+|\psi_j^\theta(X_\rho^\delta)|^2\right]d\rho\right)\nonumber
\end{align}
where $C_0$ is a constant not dependent on $\delta$. Recall now that by assuming Condition \ref{A:Assumption3}, for every $i\in\mathbb{N}$ we have $\psi_i^\theta\in\mathcal{A}_{\eta}^{\theta}$. This implies that there exists finite constants that may depend on $i, T$ and $\theta$ such that
\[
\sup_{\delta\in(0,1),\rho\in[0,T]}\mathbb E_\theta\left[|\psi_i^\theta(X_\rho^\delta)|^2\right]\leq C(\psi_{i},T,\theta)
\]
Noticing that
\[
\frac{1}{\delta}\int_{0}^{s}e^{-\frac{\lambda_i^{\theta}}{\delta}(s-\rho)}d\rho=\frac{1}{\lambda_i^{\theta}}
\left(1-e^{-\frac{\lambda_i^{\theta}}{\delta}s}\right)\leq \frac{1}{\lambda_i^{\theta}}\ ,
\]
that for $\delta$ sufficiently small $\frac{\lambda_i^{\theta}+\lambda_j^{\theta}}{\delta}e^{-\frac{\lambda_i^{\theta}+\lambda_j^{\theta}}{\delta}s}\leq 1$, and recalling  Condition \ref{A:Assumption3}, it follows that the required bound for the first statement follows with
the constant
\begin{equation}
C_{i,j}(T,\theta)=C_{0}\left(\frac{\left|\pi^{\theta}_{0}[\psi_{i}^{\theta}]\pi^{\theta}_{0}[\psi_{j}^{\theta}]\right|}{\lambda_i^{\theta}+\lambda_j^{\theta}}+
  \left(C(\psi_{i},T,\theta)+C(\psi_{j},T,\theta)\right)\left(\frac{1}{\lambda_i^{\theta}}+\frac{1}{\lambda_j^{\theta}}\right)\right)\label{Eq:BoundedTerm}
\end{equation}
The second statement is obtained by adding and subtracting the terms $\left<h_{\theta},\psi_i^{\theta}\right>_{\theta}$ and $\left<h_{\theta},\psi_j^{\theta}\right>_{\theta}$ in the products of the last integral of (\ref{Eq:UsedForBoundPsi}) and then using Theorem \ref{T:FilterConvergence1}.
\end{proof}

\begin{lemma}\label{L:UsedForTightness}
Assume Conditions \ref{A:Assumption1} and \ref{A:Assumption3} and that $\sum_{i,j=1}^{\infty}\left|\left<h_{\theta},\psi_i^{\theta}\right>_{\theta}\left<h_{\theta},\psi_j^{\theta}\right>_{\theta}\right|C_{i,j}(T,\theta)<\infty$, where
$C_{i,j}(T,\theta)$ is given by (\ref{Eq:BoundedTerm}) in Lemma \ref{L:UsedForTightnessPsi}. For any $0<T<\infty$ and for any $\theta\in\Theta$, we have that there exists a constant $C<\infty$ that does not depend on $\delta$ and $\delta_{0}<\infty $ such that
\[
\sup_{\delta\in(0,\delta_{0})}\mathbb{E}_{\alpha}\left[\frac{1}{\delta}\int_{0}^{T}\left|\pi_s^{\delta,\theta}[\tilde h_{\theta}]\right|^{2}ds\right]<C
\]
\end{lemma}
\begin{proof}
Recalling that
\[
\pi_s^{\delta,\theta}[\tilde h_{\theta}]=\sum_{i=1}^{\infty}\left<h_{\theta},\psi_i^{\theta}\right>_{\theta} \pi_s^{\delta,\theta}[\psi_i^{\theta}],
\]
we obtain
\begin{align}
\sup_{\delta\in(0,\delta_{0})}\frac{1}{\delta}\int_0^T\mathbb E_\alpha\left|\pi_s^{\delta,\theta}[\tilde h_{\theta}]\right|^{2}ds&=\sup_{\delta\in(0,\delta_{0})}\int_0^T\sum_{i,j=1}^{\infty}\left<h_{\theta},\psi_i^{\theta}\right>_{\theta}\left<h_{\theta},\psi_j^{\theta}\right>_{\theta} \frac{1}{\delta}\mathbb E_\alpha\left[\pi_s^{\delta,\theta}[\psi_i^{\theta}]\pi_s^{\delta,\theta}[\psi_j^{\theta}]\right]ds\nonumber\\
&\leq \sum_{i,j=1}^{\infty}\left|\left<h_{\theta},\psi_i^{\theta}\right>_{\theta}\left<h_{\theta},\psi_j^{\theta}\right>_{\theta} \right|\sup_{\delta\in(0,\delta_{0})}\int_0^T\frac{1}{\delta}\left|\mathbb E_\alpha\pi_s^{\delta,\theta}[\psi_i^{\theta}]\pi_s^{\delta,\theta}[\psi_j^{\theta}]\right|ds\nonumber\\
&\leq T\sum_{i,j=1}^{\infty}\left|\left<h_{\theta},\psi_i^{\theta}\right>_{\theta}\left<h_{\theta},\psi_j^{\theta}\right>_{\theta}\right|C_{i,j}(T,\theta) \nonumber\\
&<\infty\ ,
\nonumber
\end{align}
and so the constant is $C=T\sum_{i,j=1}^{\infty}\left|\left<h_{\theta},\psi_i^{\theta}\right>_{\theta}\left<h_{\theta},\psi_j^{\theta}\right>_{\theta}\right|C_{i,j}(T,\theta)$.
\end{proof}

\begin{lemma}\label{L:ConvergenceInProb1}
Assume Conditions \ref{A:Assumption1} and \ref{A:Assumption3} . For any $0<T<\infty$ and $\theta\in \Theta$ we have
\[
\sup_{t\in[0,T]}\mathbb{E}_{\alpha}\left|\frac{1}{\sqrt{\delta}}\int_0^t e^{-\frac{\lambda_i^{\theta}(t-s)}{\delta}}
\left(\pi_s^{\delta,\theta}[h_{\theta}\psi_i^{\theta}]-\left<h_{\theta},\psi_i^{\theta}\right>_{\theta}-\pi_s^{\delta,\theta}[h_{\theta}]\pi_s^{\delta,\theta}[\psi_i^{\theta}]\right)
d\nu_s^{\delta,\alpha}\right|^{2}\rightarrow 0, \textrm{ as }\delta\downarrow 0
\]
\end{lemma}
\begin{proof}
Due to It\^{o} isometry we have
\begin{align}
&\mathbb{E}_{\alpha}\left|\frac{1}{\sqrt{\delta}}\int_0^t e^{-\frac{\lambda_i^{\theta}(t-s)}{\delta}}
\left(\pi_s^{\delta,\theta}[h_{\theta}\psi_i^{\theta}]-\left<h_{\theta},\psi_i^{\theta}\right>_{\theta}-\pi_s^{\delta,\theta}[h_{\theta}]\pi_s^{\delta,\theta}[\psi_i^{\theta}]\right)
d\nu_s^{\delta,\alpha}\right|^{2} \nonumber\\
&=\mathbb{E}_{\alpha}\left|\frac{1}{\sqrt{\delta}}\int_0^t e^{-\frac{\lambda_i^{\theta}(t-s)}{\delta}}
\left(\pi_s^{\delta,\theta}[h_{\theta}\psi_i^{\theta}]-\left<h_{\theta},\psi_i^{\theta}\right>_{\theta}-\left(\pi_s^{\delta,\theta}[h_{\theta}]-\bar h_\theta\right)\pi_s^{\delta,\theta}[\psi_i^{\theta}]-\bar h_\theta\pi_s^{\delta,\theta}[\psi_i^{\theta}]\right)
d\nu_s^{\delta,\alpha}\right|^{2} \nonumber\\
&\leq 3 \left[\frac{1}{\delta} \int_0^t e^{-2\frac{\lambda_i^{\theta}(t-s)}{\delta}}
\mathbb{E}_{\alpha}\left|\pi_s^{\delta,\theta}[h_{\theta}\psi_i^{\theta}]-\left<h_{\theta},\psi_i^{\theta}\right>_{\theta}\right|^{2} ds+\frac{C_h^2}{\delta} \int_0^t e^{-2\frac{\lambda_i^{\theta}(t-s)}{\delta}}
\mathbb{E}_{\alpha}\left|\pi_s^{\delta,\theta}[\psi_i^{\theta}]\right|^{2} ds\right.\nonumber\\
&\hspace{0.5cm} \left.
+\frac{1}{\delta} \int_0^t e^{-2\frac{\lambda_i^{\theta}(t-s)}{\delta}}
\mathbb{E}_{\alpha}\left|\pi_s^{\delta,\theta}[h_{\theta}]-\bar{h}_{\theta}\right|^{2} ds\right]\ .
\label{Eq:lastLemmaB3}
\end{align}

Noticing that
\[
\sup_{t\in[0,T]}\frac{1}{\delta} \int_0^t e^{-2\frac{\lambda_i^{\theta}(t-s)}{\delta}}ds=\sup_{t\in[0,T]}\frac{1}{\lambda_{i}^\theta}\left(1-e^{-2\frac{\lambda_i^{\theta}}{\delta} t}\right)\leq \frac{1}{\lambda_{i}^\theta}\ .
\]
the statement of the lemma follows by Theorem \ref{T:FilterConvergence1} and dominated convergence theorem to equation \eqref{Eq:lastLemmaB3}. Notice that  dominated convergence theorem is applicable since we can apply Lemma \ref{L:parameterMismatch} to the integrands and notice that the integrands are expectations of functions in $\mathcal A_\eta^\theta$.
\end{proof}
\begin{lemma}\label{L:MismatchTerm}
Assume the Conditions of Lemma \ref{L:UsedForTightness}. For any $0<T<\infty$, and for any $\theta\in\Theta$ we have in $\mathbb{P}_{\alpha}$ probability and uniformly in $t\in[0,T]$  that
\[\int_0^t\frac{1}{\sqrt\delta}\left( \pi_s^{\delta,\theta}[\tilde h_{\theta}]\right)\left(\pi_s^{\delta,\theta}[h_{\theta}]-\pi_s^{\delta,\alpha}[h_{\alpha}]\right)ds\rightarrow 0, \quad\text{ as  }\delta \downarrow 0\]
\end{lemma}
\begin{proof}
First we notice that
\[
\pi_s^{\delta,\theta}[h_{\theta}]-\pi_s^{\delta,\alpha}[h_{\alpha}]=\left(\bar{h}_{\theta}-\bar{h}_{\alpha}\right)+\left(\pi_s^{\delta,\theta}[\tilde h_{\theta}]-\pi_s^{\delta,\alpha}[\tilde h_{\alpha}]\right)
\]
Using the Cauchy inequality ($ab\leq a^{2}/2+b^{2}/2$ for all $a,b\in\mathbb R$), this implies that
\begin{align}
&\mathbb{E}_{\alpha}\sup_{t\in[0,T]}\left|\int_0^t\frac{1}{\sqrt\delta}\left( \pi_s^{\delta,\theta}[\tilde h_{\theta}]\right)\left(\pi_s^{\delta,\theta}[h_{\theta}]-\pi_s^{\delta,\alpha}[h_{\alpha}]\right)ds\right|\leq \nonumber\\
&\leq
\left|\bar{h}_{\theta}-\bar{h}_{\alpha}\right|\mathbb{E}_{\alpha}\sup_{t\in[0,T]}\left|\int_0^t\frac{1}{\sqrt\delta}\left( \pi_s^{\delta,\theta}[\tilde h_{\theta}]\right)ds\right|+\mathbb{E}_{\alpha}\sup_{t\in[0,T]}\left|\int_0^t\frac{1}{\sqrt\delta}\left( \pi_s^{\delta,\theta}[\tilde h_{\theta}]\right)\left(\pi_s^{\delta,\theta}[\tilde h_{\theta}]-\pi_s^{\delta,\alpha}[\tilde h_{\alpha}]\right)ds\right|\nonumber\\
&\leq
\left|\bar{h}_{\theta}-\bar{h}_{\alpha}\right|\mathbb{E}_{\alpha}\sup_{t\in[0,T]}\left|\int_0^t\frac{1}{\sqrt\delta}\left( \pi_s^{\delta,\theta}[\tilde h_{\theta}]\right)ds\right|+\mathbb{E}_{\alpha}\sup_{t\in[0,T]}\int_0^t\frac{2}{\sqrt\delta}\left| \pi_s^{\delta,\theta}[\tilde h_{\theta}]\right|^{2}ds
+\mathbb{E}_{\alpha}\sup_{t\in[0,T]}\int_0^t\frac{1}{\sqrt\delta}\left| \pi_s^{\delta,\alpha}[\tilde h_{\alpha}]\right|^{2}ds
\nonumber\\
&\qquad\leq
\left|\bar{h}_{\theta}-\bar{h}_{\alpha}\right|\mathbb{E}_{\alpha}\sup_{t\in[0,T]}\left|\int_0^t\frac{1}{\sqrt\delta}\left( \pi_s^{\delta,\theta}[\tilde h_{\theta}]\right)ds\right|+3C\sqrt{\delta}
\end{align}
where in the last step we used the bound from Lemma  \ref{L:UsedForTightness}. Since $3C\sqrt\delta$ clearly goes to zero as $\delta\downarrow 0$, it remains to show that the first term will also go to zero as $\delta\downarrow 0$. Namely, it remains to show that
\[
\lim_{\delta\downarrow 0}\mathbb{E}_{\alpha}\sup_{t\in[0,T]}\left|\int_0^t\frac{1}{\sqrt\delta}\left( \pi_s^{\delta,\theta}[\tilde h_{\theta}]\right)ds\right|=0.
\]
Using similar computations as in the proof of Lemma \ref{L:UsedForTightness}, we notice that

\begin{align}
&\mathbb{E}_{\alpha}\sup_{t\in[0,T]}\left|\frac{1}{\sqrt{\delta}} \int_{0}^{t}\pi_s^{\delta,\theta}[\tilde h_{\theta}] ds\right|\leq \left| \frac{1}{\sqrt{\delta}}\delta\sum_{i=1}^{\infty}\frac{\left<h_{\theta},\psi_i^{\theta}\right>_{\theta}
}{\lambda_i^{\theta}}\pi^{\theta}_{0}[\psi_{i}^{\theta}]
\sup_{t\in[0,T]}\frac{\lambda_i^{\theta}}{\delta}\int_{0}^{t}e^{-\frac{\lambda_i^{\theta}}{\delta}s}ds\right|\nonumber\\
&\hspace{0.5cm}+\mathbb{E}_{\alpha}\sup_{t\in[0,T]}\left|\sum_{i=1}^{\infty}\left<h_{\theta},\psi_i^{\theta}\right>_{\theta}\frac{1}{\sqrt{\delta}}\int_{0}^{t}\int_{0}^{s}e^{-\frac{\lambda_i^{\theta}}{\delta}(s-\rho)}
\left(\pi^{\delta,\theta}_{\rho}[h_{\theta}\psi_{i}^{\theta}]-\pi^{\delta,\theta}_{\rho}[\psi_{i}^{\theta}]\pi^{\delta,\theta}_{\rho}[h_{\theta}]   \right)d\nu^{\delta,\alpha}_{\rho} ds\right|
\nonumber\\
&\hspace{0.5cm}+\mathbb{E}_{\alpha}\sup_{t\in[0,T]}\left|\sum_{i=1}^{\infty}\left<h_{\theta},\psi_i^{\theta}\right>_{\theta}\frac{1}{\sqrt{\delta}}\int_{0}^{t}\int_{0}^{s}e^{-\frac{\lambda_i^{\theta}}{\delta}(s-\rho)}
\left(\pi^{\delta,\theta}_{\rho}[h_{\theta}\psi_{i}^{\theta}]-\pi^{\delta,\theta}_{\rho}[\psi_{i}^{\theta}]\pi^{\delta,\theta}_{\rho}[h_{\theta}]   \right)\left(\pi^{\delta,\theta}_{\rho}[h_{\theta}]-\pi^{\delta,\alpha}_{\rho}[h_{\alpha}]\right) d\rho ds\right|\nonumber\\
&\leq 2\sqrt{\delta}\sum_{i=1}^{\infty}\left|\frac{\left<h_{\theta},\psi_i^{\theta}\right>_{\theta}
}{\lambda_i^{\theta}}\pi^{\theta}_{0}[\psi_{i}^{\theta}]\right|
\nonumber\\
&\hspace{0.5cm}+\sum_{i=1}^{\infty}\left|\left<h_{\theta},\psi_i^{\theta}\right>_{\theta}\right|\int_{0}^{T}\mathbb{E}_{\alpha}\frac{1}{\sqrt{\delta}}\left|\int_{0}^{s}e^{-\frac{\lambda_i^{\theta}}{\delta}(s-\rho)}
\left(\pi^{\delta,\theta}_{\rho}[h_{\theta}\psi_{i}^{\theta}]-\left<h_{\theta},\psi_i^{\theta}\right>_{\theta}-\pi^{\delta,\theta}_{\rho}[\psi_{i}^{\theta}]\pi^{\delta,\theta}_{\rho}[h_{\theta}]   \right)d\nu^{\delta,\alpha}_{\rho} \right| ds
\nonumber\\
&\hspace{0.5cm}+\sum_{i=1}^{\infty}\left|\left<h_{\theta},\psi_i^{\theta}\right>_{\theta}\right|^{2}\frac{1}{\sqrt{\delta}}\mathbb{E}_{\alpha}\sup_{t\in[0,T]}\left|\int_{0}^{t}\int_{0}^{s}e^{-\frac{\lambda_i^{\theta}}{\delta}(s-\rho)}
d\nu^{\delta,\alpha}_{\rho} ds\right|
\nonumber\\
&\hspace{0.5cm}+\sqrt{\delta}\sum_{i=1}^{\infty}\frac{\left|\left<h_{\theta},\psi_i^{\theta}\right>_{\theta}\right|}{\lambda_{i}^{\theta}}\frac{\lambda_{i}^{\theta}}{\delta}\int_{0}^{T}\int_{0}^{s}e^{-\frac{\lambda_i^{\theta}}{\delta}(s-\rho)}
\mathbb{E}_{\alpha}\left|\left(\pi^{\delta,\theta}_{\rho}[h_{\theta}\psi_{i}^{\theta}]-\pi^{\delta,\theta}_{\rho}[\psi_{i}^{\theta}]\pi^{\delta,\theta}_{\rho}[h_{\theta}]   \right)\left(\pi^{\delta,\theta}_{\rho}[h_{\theta}]-\pi^{\delta,\alpha}_{\rho}[h_{\alpha}]\right)\right| d\rho ds\nonumber
\end{align}
Clearly, the first term goes to zero as $\delta\downarrow 0$.  Similarly, the fourth term also goes to zero as $\delta\downarrow 0$ and this follows by Condition \ref{A:Assumption3}(i)-(iii). By Lemma \ref{L:ConvergenceInProb1}, the second term can also be shown to go to zero. So, it essentially remains to treat the third term. For this purpose, we rcall that the solution to the equation (\ref{Eq:OUSDE1}), $\Xi^{\delta,i}_{t}$, is given by (\ref{Eq:OUSDE2}),
which is normally distributed with mean zero and variance $\frac{1}{2\lambda_i^{\theta}}\left(1-e^{-\frac{\lambda_i^{\theta}t}{\delta}}\right)$.
Hence,  the third term in question can be written as
\begin{align}
&\sum_{i=1}^{\infty}\left|\left<h_{\theta},\psi_i^{\theta}\right>_{\theta}\right|^{2}\frac{1}{\sqrt{\delta}}\mathbb{E}_{\alpha}\sup_{t\in[0,T]}\left|\int_{0}^{t}\int_{0}^{s}e^{-\frac{\lambda_i^{\theta}}{\delta}(s-\rho)}
d\nu^{\delta,\alpha}_{\rho} ds\right|=
\sum_{i=1}^{\infty}\left|\left<h_{\theta},\psi_i^{\theta}\right>_{\theta}\right|^{2}\mathbb{E}_{\alpha}\sup_{t\in[0,T]}\left|\int_{0}^{t}\Xi^{\delta,i}_{s} ds\right|\nonumber\\
&\hspace{6cm}\leq \sqrt{\delta} \left\{\sum_{i=1}^{\infty}\frac{\left|\left<h_{\theta},\psi_i^{\theta}\right>_{\theta}\right|^{2}}{\lambda_i^{\theta}
}\left[\sqrt{\delta} \mathbb{E}_{\alpha}\sup_{t\in[0,T]}\left|\Xi^{\delta,i}_{t}\right|+\mathbb{E}_{\alpha}\sup_{t\in[0,T]}\left|\nu_t^{\delta,\alpha}\right| \right]\right\},
\end{align}
and it is easy to see that this term goes to zero as $\delta\downarrow 0$. This completes the proof of the lemma.
\end{proof}

\begin{lemma}
\label{L:R2convergence}
Assume Conditions \ref{A:Assumption1} and \ref{A:Assumption3} and that
\[
\sum_{i,j=1}^{\infty}\frac{\left|\left<h_{\theta},\psi_i^{\theta}\right>_{\theta} \left<h_{\theta},\psi_j^{\theta}\right>_{\theta}\right|\left(C(\psi_{i},T,\theta)+C(\psi_{j},T,\theta)\right)}{\lambda_i^{\theta}+\lambda_j^{\theta}}
 <\infty
\]
where $C(\psi_{i},T,\theta)$ is a constant such that \[
\sup_{\delta\in(0,1),\rho\in[0,T]}\mathbb E_\theta\left[|\psi_i^\theta(X_\rho^\delta)|^2\right]\leq C(\psi_{i},T,\theta)\ ,
\]
which is statement of equation \eqref{Eq:BoundedTerm01}. Then, the term $R^{2,\delta}_t$ from equation \eqref{Eq:ForLastTerm} converges to zero in mean-square sense uniformly in $t\in[0,T]$,
\[\lim_{\delta\downarrow 0}\mathbb{E}_{\alpha}\sup_{t\in[0,T]}\left|R^{2,\delta}_{t}\right|^{2}=0.\]
\end{lemma}
\begin{proof}
Since $R_t^{2,\delta}$ is a martingale, by Doob's inequality we have
\[\mathbb{E}_{\alpha}\sup_{t\in[0,T]}\left|R^{2,\delta}_{t}\right|^{2}\leq 4 \mathbb{E}_{\alpha}\left[R^{2,\delta}_T\right]^{2}\ ,\]
and it follows by the Cauchy inequality (i.e. $ab\leq a^{2}/2+b^{2}/2$ for any $a,b\in\mathbb R$) and then It\^{o} isometry, that
\begin{align}
&\mathbb{E}_{\alpha}\left[R^{2,\delta}_{T}\right]^{2}= \sum_{i,j=1}^{\infty}\left<h_{\theta},\psi_i^{\theta}\right>_{\theta} \left<h_{\theta},\psi_j^{\theta}\right>_{\theta}   \frac{1}{\delta}\mathbb{E}_{\alpha}\left[ \int_{0}^{T}\left(\int_0^s e^{-\frac{\lambda_i^{\theta}(s-\rho)}{\delta}}\left(\pi_\rho^{\delta,\theta}[h_{\theta}\psi_i^{\theta}]-\left<h_{\theta},\psi_i^{\theta}\right>_{\theta}
-\pi_\rho^{\delta,\theta}[h_{\theta}]\pi_\rho^{\delta,\theta}[\psi_i^{\theta}]
\right)d\nu_\rho^{\delta,\theta}\right)\times \right.\nonumber\\
&\hspace{1cm}\left.\times \left(\int_0^s e^{-\frac{\lambda_j^{\theta}(s-\rho)}{\delta}}\left(\pi_\rho^{\delta,\theta}[h_{\theta}\psi_j^{\theta}]-\left<h_{\theta},\psi_j^{\theta}\right>_{\theta}
-\pi_\rho^{\delta,\theta}[h_{\theta}]\pi_\rho^{\delta,\theta}[\psi_j^{\theta}]
\right)d\nu_\rho^{\delta,\theta}\right)ds\right]\nonumber\\
&\leq \sum_{i,j=1}^{\infty}\left<h_{\theta},\psi_i^{\theta}\right>_{\theta} \left<h_{\theta},\psi_j^{\theta}\right>_{\theta}  \frac{1}{\delta}\left[ \int_{0}^{T}\Bigg(\int_0^s e^{-\frac{\left(\lambda_i^{\theta}+\lambda_j^{\theta}\right)(s-\rho)}{\delta}}\mathbb{E}_{\alpha}\left(\pi_\rho^{\delta,\theta}[h_{\theta}\psi_i^{\theta}]-\left<h_{\theta},\psi_i^{\theta}\right>_{\theta}
-\pi_\rho^{\delta,\theta}[h_{\theta}]\pi_\rho^{\delta,\theta}[\psi_i^{\theta}]
\right)\times \right.\nonumber\\
&\hspace{1cm}\left.\times \left(\pi_\rho^{\delta,\theta}[h_{\theta}\psi_j^{\theta}]-\left<h_{\theta},\psi_j^{\theta}\right>_{\theta}
-\pi_\rho^{\delta,\theta}[h_{\theta}]\pi_\rho^{\delta,\theta}[\psi_j^{\theta}]
\right)d\rho\Bigg) ds\right]\nonumber\\
&+ |2C_{h}|^{2}\sum_{i,j=1}^{\infty}\left|\left<h_{\theta},\psi_i^{\theta}\right>_{\theta} \left<h_{\theta},\psi_j^{\theta}\right>_{\theta} \right| \frac{1}{\delta}\left[ \int_{0}^{T}\left(\int_0^s e^{-\frac{\left(\lambda_i^{\theta}+\lambda_j^{\theta}\right)(s-\rho)}{\delta}}\mathbb{E}_{\alpha}\left|\pi_\rho^{\delta,\theta}[h_{\theta}\psi_i^{\theta}]-\left<h_{\theta},\psi_i^{\theta}\right>_{\theta}
-\pi_\rho^{\delta,\theta}[h_{\theta}]\pi_\rho^{\delta,\theta}[\psi_i^{\theta}]
\right|\times \right.\right.\nonumber\\
&\hspace{1cm}\left.\times \left|\pi_\rho^{\delta,\theta}[h_{\theta}\psi_j^{\theta}]-\left<h_{\theta},\psi_j^{\theta}\right>_{\theta}
-\pi_\rho^{\delta,\theta}[h_{\theta}]\pi_\rho^{\delta,\theta}[\psi_j^{\theta}]
\right|d\rho\Bigg) ds\right]\nonumber\\
&\leq \frac{1+|C_{h}|^{2}}{2}\sum_{i,j=1}^{\infty}\int_{0}^{T}\left|\left<h_{\theta},\psi_i^{\theta}\right>_{\theta} \left<h_{\theta},\psi_j^{\theta}\right>_{\theta}\right|  \frac{1}{\delta}\left[ \int_0^s e^{-\frac{\left(\lambda_i^{\theta}+\lambda_j^{\theta}\right)(s-\rho)}{\delta}}\Bigg(\mathbb{E}_{\alpha}\left(\pi_\rho^{\delta,\theta}[h_{\theta}\psi_i^{\theta}]-\left<h_{\theta},\psi_i^{\theta}\right>_{\theta}
-\pi_\rho^{\delta,\theta}[h_{\theta}]\pi_\rho^{\delta,\theta}[\psi_i^{\theta}]
\right)^{2}+ \right.\nonumber\\
&\hspace{1cm}\left.\left.+ \mathbb{E}_{\alpha}\left(\pi_\rho^{\delta,\theta}[h_{\theta}\psi_j^{\theta}]-\left<h_{\theta},\psi_j^{\theta}\right>_{\theta}
-\pi_\rho^{\delta,\theta}[h_{\theta}]\pi_\rho^{\delta,\theta}[\psi_j^{\theta}]
\right)^{2} d\rho\right) ds\right]\label{Eq:R2Bound1}
\end{align}
Now we want to apply dominated convergence theorem equation \eqref{Eq:R2Bound1} in order to argue that the upper bound of the last inequality goes to zero as $\delta\downarrow 0$. First of all, we notice that by Lemma \ref{L:parameterMismatch}, we have the following bound for the integrand
\begin{align}
& \left|\left<h_{\theta},\psi_i^{\theta}\right>_{\theta} \left<h_{\theta},\psi_j^{\theta}\right>_{\theta}\right|  \frac{1}{\delta}\left(\int_0^s e^{-\frac{\left(\lambda_i^{\theta}+\lambda_j^{\theta}\right)(s-\rho)}{\delta}}\mathbb{E}_{\alpha}\left(\pi_\rho^{\delta,\theta}[h_{\theta}\psi_i^{\theta}]-\left<h_{\theta},\psi_i^{\theta}\right>_{\theta}
-\pi_\rho^{\delta,\theta}[h_{\theta}]\pi_\rho^{\delta,\theta}[\psi_i^{\theta}]
\right)^{2}+ \right.\nonumber\\
&\hspace{1cm}\left.+ \mathbb{E}_{\alpha}\left(\pi_\rho^{\delta,\theta}[h_{\theta}\psi_j^{\theta}]-\left<h_{\theta},\psi_j^{\theta}\right>_{\theta}
-\pi_\rho^{\delta,\theta}[h_{\theta}]\pi_\rho^{\delta,\theta}[\psi_j^{\theta}]
\right)^{2} d\rho\right) \nonumber\\
&\leq C_{0}\left|\left<h_{\theta},\psi_i^{\theta}\right>_{\theta} \left<h_{\theta},\psi_j^{\theta}\right>_{\theta}\right|  \frac{1}{\delta}\int_0^s e^{-\frac{\left(\lambda_i^{\theta}+\lambda_j^{\theta}\right)(s-\rho)}{\delta}}
\left[\mathbb{E}_{\theta}\left(\left|\psi_i^{\theta}\left(X^{\delta}_{\rho}\right)\right|^{2}+\left|\psi_j^{\theta}\left(X^{\delta}_{\rho}\right)\right|^{2}\right)
\right]
 d\rho \label{Eq:R2Bound2}
\end{align}
Recall now that by assuming Condition \ref{A:Assumption3}, for every $i\in\mathbb{N}$ we have $\psi_i^\theta\in\mathcal{A}_{\eta}^{\theta}$. This implies that there exists finite constants that may depend on $i, T$ and $\theta$ such that
\[
\sup_{\delta\in(0,1),\rho\in[0,T]}\mathbb E_\theta\left[|\psi_i^\theta(X_\rho^\delta)|^2\right]\leq C(\psi_{i},T,\theta)\ ,
\]
Noticing that
\[
\frac{1}{\delta}\int_{0}^{s}e^{-\frac{\lambda_i^{\theta}+\lambda_j^{\theta}}{\delta}(s-\rho)}d\rho=\frac{1}{\lambda_i^{\theta}+\lambda_j^{\theta}}
\left(1-e^{-\frac{\lambda_i^{\theta}+\lambda_j^{\theta}}{\delta}s}\right)\leq \frac{1}{\lambda_i^{\theta}+\lambda_j^{\theta}}
\]
we can then continue bounding (\ref{Eq:R2Bound2}) by the term

\begin{align}
&\leq C_{0}\left|\left<h_{\theta},\psi_i^{\theta}\right>_{\theta} \left<h_{\theta},\psi_j^{\theta}\right>_{\theta}\right|  \frac{1}{\delta}\int_0^s e^{-\frac{\left(\lambda_i^{\theta}+\lambda_j^{\theta}\right)(s-\rho)}{\delta}}
\left[C(\psi_{i},T,\theta)+C(\psi_{j},T,\theta)\right]
 d\rho\nonumber\\
&\leq C_{0}\left[\frac{\left|\left<h_{\theta},\psi_i^{\theta}\right>_{\theta} \left<h_{\theta},\psi_j^{\theta}\right>_{\theta}\right|}{\lambda_i^{\theta}+\lambda_j^{\theta}}
\left(C(\psi_{i},T,\theta)+C(\psi_{j},T,\theta)\right)  \right]
  \label{Eq:R2Bound3}
\end{align}
Hence, the summands in (\ref{Eq:R2Bound1}) are dominated by terms that are summable and is finite irrespective of $\delta\in(0,1)$. Secondly, by Theorem  \ref{T:FilterConvergence1}  we know that for each $i\geq 1$ there is the limit
\[
\lim_{\delta\downarrow 0}\mathbb{E}_{\alpha}\left(\pi_\rho^{\delta,\theta}[h_{\theta}\psi_j^{\theta}]-\left<h_{\theta},\psi_j^{\theta}\right>_{\theta}
-\pi_\rho^{\delta,\theta}[h_{\theta}]\pi_\rho^{\delta,\theta}[\psi_j^{\theta}]
\right)^{2}=0\ .
\]
Hence, by dominated convergence we have established that (\ref{Eq:R2Bound1}) goes to zero in probability, and then it follows that
\begin{equation*}
\lim_{\delta\downarrow 0}\mathbb{E}_{\alpha}\sup_{t\in[0,T]}\left|R^{2,\delta}_{t}\right|^{2}=0.
\end{equation*}

\end{proof}

\begin{lemma}
\label{L:R3convergence}
Assume the Conditions of Lemma \ref{L:UsedForTightness}.
Then, the term $R^{3,\delta}_t$ from equation \eqref{Eq:ForLastTerm} converges to zero in mean-square sense uniformly in $t\in[0,T]$,
\[\lim_{\delta\downarrow 0}\mathbb{E}_{\alpha}\sup_{t\in[0,T]}\left|R^{3,\delta}_{t}\right|^{2}=0.\]
\end{lemma}
\begin{proof}
 We have
\begin{align}
R^{3,\delta}_{t}&=\sum_{i=1}^{\infty}\left|\left<h_{\theta},\psi_i^{\theta}\right>_{\theta}\right|^{2}\frac{1}{\sqrt\delta}\int_{0}^{t}\left[\int_0^s e^{-\frac{\lambda_i^{\theta}(s-\rho)}{\delta}}\left(\pi_\rho^{\delta,\theta}[h_{\theta}]-\pi_\rho^{\delta,\alpha}[h_{\alpha}]\right)d\rho\right]d\nu_s^{\delta,\alpha}\nonumber\\
&=(\bar{h}_{\theta}-\bar{h}_{\alpha})\sum_{i=1}^{\infty}\left|\left<h_{\theta},\psi_i^{\theta}\right>_{\theta}\right|^{2}\frac{1}{\sqrt\delta}\int_{0}^{t}\left[\int_0^s e^{-\frac{\lambda_i^{\theta}(s-\rho)}{\delta}}d\rho\right]d\nu_s^{\delta,\alpha}\nonumber\\
&\quad+\sum_{i=1}^{\infty}\left|\left<h_{\theta},\psi_i^{\theta}\right>_{\theta}\right|^{2}\frac{1}{\sqrt\delta}\int_{0}^{t}\left[\int_0^s e^{-\frac{\lambda_i^{\theta}(s-\rho)}{\delta}}\left(\pi_\rho^{\delta,\theta}[\tilde h_{\theta}]-\pi_\rho^{\delta,\alpha}[\tilde h_{\alpha}]\right)d\rho\right]d\nu_s^{\delta,\alpha}\nonumber\\
&=\sqrt{\delta}(\bar{h}_{\theta}-\bar{h}_{\alpha})\sum_{i=1}^{\infty}\frac{\left|\left<h_{\theta},\psi_i^{\theta}\right>_{\theta}\right|^{2}}{\lambda_i^{\theta}}\int_{0}^{t}\left(1- e^{-\frac{\lambda_i^{\theta}s}{\delta}}\right)d\nu_s^{\delta,\alpha}\nonumber\\
&\quad+\sum_{i=1}^{\infty}\left|\left<h_{\theta},\psi_i^{\theta}\right>_{\theta}\right|^{2}\frac{1}{\sqrt\delta}\int_{0}^{t}\left[\int_0^s e^{-\frac{\lambda_i^{\theta}(s-\rho)}{\delta}}\left(\pi_\rho^{\delta,\theta}[\tilde h_{\theta}]-\pi_\rho^{\delta,\alpha}[\tilde h_{\alpha}]\right)d\rho\right]d\nu_s^{\delta,\alpha}
\end{align}
By  applying Doob's inequality followed by the Cauchy inequality and then Jensen's inequality,  we have
\begin{align}
\mathbb{E}_{\alpha}\sup_{t\in[0,T]}\left|R^{3,\delta}_{t}\right|^{2}
&\leq \delta 2(\bar{h}_{\theta}-\bar{h}_{\alpha})^{2}\sum_{i,j=1}^{\infty}\frac{\left|\left<h_{\theta},\psi_i^{\theta}\right>_{\theta}\left<h_{\theta},\psi_j^{\theta}\right>_{\theta}\right|^{2}}{\lambda_i^{\theta}+\lambda_j^{\theta}}
\int_{0}^{T}\left(1- e^{-\frac{\lambda_i^{\theta}s}{\delta}}\right)\left(1- e^{-\frac{\lambda_j^{\theta}s}{\delta}}\right)ds\nonumber\\
&\quad+2\sum_{i,j=1}^{\infty}\left|\left<h_{\theta},\psi_i^{\theta}\right>_{\theta}\left<h_{\theta},\psi_j^{\theta}\right>_{\theta}\right|^{2}\frac{1}{\delta}\mathbb{E}_{\alpha}\int_{0}^{T}\left(\int_0^s e^{-\frac{\lambda_i^{\theta}(s-\rho)}{\delta}}\left(\pi_\rho^{\delta,\theta}[\tilde h_{\theta}]-\pi_\rho^{\delta,\alpha}[\tilde h_{\alpha}]\right)d\rho\right)\times\nonumber\\
&\hspace{6cm}\times \left(\int_0^s e^{-\frac{\lambda_j^{\theta}(s-\rho)}{\delta}}\left(\pi_\rho^{\delta,\theta}[\tilde h_{\theta}]-\pi_\rho^{\delta,\alpha}[\tilde h_{\alpha}]\right)d\rho\right)ds\nonumber\\
&\leq \delta 2T(\bar{h}_{\theta}-\bar{h}_{\alpha})^{2}\sum_{i,j=1}^{\infty}\frac{\left|\left<h_{\theta},\psi_i^{\theta}\right>_{\theta}\left<h_{\theta},\psi_j^{\theta}\right>_{\theta}\right|^{2}}{\lambda_i^{\theta}+\lambda_j^{\theta}}
\nonumber\\
&\quad+2\sum_{i,j=1}^{\infty}\left|\left<h_{\theta},\psi_i^{\theta}\right>_{\theta}\left<h_{\theta},\psi_j^{\theta}\right>_{\theta}\right|^{2}\int_{0}^{T}\left(\int_0^s e^{-\frac{2\lambda_1^{\theta}(s-\rho)}{\delta}}\frac{1}{\delta}\mathbb{E}_{\alpha}\left(\pi_\rho^{\delta,\theta}[\tilde h_{\theta}]-\pi_\rho^{\delta,\alpha}[\tilde h_{\alpha}]\right)^{2}d\rho\right)ds\nonumber
\end{align}
Then, by the fact that $v^{2}_{\theta}(h_{\theta})<\infty$ and $\sum_{i,j=1}^{\infty}\left|\left<h_{\theta},\psi_i^{\theta}\right>_{\theta}\left<h_{\theta},\psi_j^{\theta}\right>_{\theta}\right|^{2}<\infty$  (see Remark \ref{R:AssumptionOnVariance}), and the uniform bound from Lemma \ref{L:UsedForTightness}, we obtain that
\begin{equation*}
\lim_{\delta\downarrow 0}\mathbb{E}_{\alpha}\sup_{t\in[0,T]}\left|R^{3,\delta}_{t}\right|^{2}=0.
\end{equation*}
concluding the proof of the lemma.
\end{proof}

\small{\bibliography{refs}}

\begin{thebibliography}{}

\bibitem[Abramowitz and Stegun, 1965]{Abramowitz}
Abramowitz, M. and Stegun, I. (1965).
\newblock {\em Handbook of Mathematical Functions}.
\newblock Dover Publications, Mineola, NY.

\bibitem[Bain and Crisan, 2009]{bainCrisan}
Bain, A. and Crisan, D. (2009).
\newblock {\em Fundamentals of Stochastic Filtering}.
\newblock Springer, New York, NY.

\bibitem[Bensoussan and Blankenship, 1986]{BensoussanBlankenship1986}
Bensoussan, A. and Blankenship, G.~L. (1986).
\newblock Nonlinear filtering with homogenization.
\newblock {\em Stochastics}, 17:67--90.

\bibitem[Bensoussan et~al., 1978]{BLP}
Bensoussan, A., Lions, J., and Papanicolaou, G. (1978).
\newblock {\em Asymptotic Analysis for Periodic Structures}, volume~5 of {\em
  Studies in Mathematics and its Applications}.
\newblock North-Holland Publishing Co., Amsterdam.

\bibitem[Billingsley, 1968]{Billingsley}
Billingsley, P. (1968).
\newblock {\em Convergence of Probability Measures}.
\newblock New York, J. Willey.

\bibitem[Boyd, 1984]{boyd1984}
Boyd, J. (1984).
\newblock Asymptotic coefficients of {H}ermite function series.
\newblock {\em Journal of Computational Physics}, 54:382--410.

\bibitem[Boyd, 2000]{boyd2000}
Boyd, J. (2000).
\newblock {\em {C}hebyshev and {F}ourier Spectral Methods}.
\newblock Dover Publications, inc., Mineola, New York, 2nd edition.

\bibitem[Brogaard et~al., 2012]{hendershot}
Brogaard, J., Hendershott, T.~J., and Riordan, R. (2012).
\newblock {High Frequency Trading and Price Discovery}.
\newblock Technical report, Berkeley University.

\bibitem[Capp{\'e} et~al., 2005]{CMR2005}
Capp{\'e}, O., Moulines, E., and Ryden, T. (2005).
\newblock {\em Inference in Hidden Markov Models (Springer Series in
  Statistics)}.
\newblock Springer-Verlag New York, Inc., Secaucus, NJ, USA.

\bibitem[{Del Moral} et~al., 2001]{delMoral2001}
{Del Moral}, P., Jacod, J., and Protter, P. (2001).
\newblock The {M}onte-{C}arlo method for filtering with discrete-time
  observations.
\newblock {\em Probability Theory and Related Fields}, 120:346 to 368.

\bibitem[Ethier and Kurtz, 1986]{EthierKurtz1986}
Ethier, S. and Kurtz, T. (1986).
\newblock {\em {M}arkov Processes: Characterization and Convergence}.
\newblock Wiley, Hoboken, NJ.

\bibitem[Ichihara, 2004]{Ichihara2004}
Ichihara, N. (2004).
\newblock Homogenization problem for stochastic partial differential equations
  of zakai type.
\newblock {\em Stochastics and Stochastics Reports}, 76:243--266.

\bibitem[Imkeller et~al., 2013]{ImkellerSriPerkowskiYeong2012}
Imkeller, P., Namachchivaya, N.~S., Perkowski, N., and Yeong, H.~C. (2013).
\newblock Dimensional reduction in nonlinear filtering: a homogenization
  approach.
\newblock {\em Annals of Applied Probability}, 23(6):2290--2326.

\bibitem[James and Gland, 1995]{JamesGland1995}
James, M.~R. and Gland, F.~L. (1995).
\newblock Consistent parameter estimation for partially observed diffusions
  with small noise.
\newblock {\em Applied Mathematics and Optimization}, 32:47--72.

\bibitem[Kallianpur, 1980]{Kallianpur}
Kallianpur, G. (1980).
\newblock {\em Stochastic Filtering Theory}.
\newblock Springer, Berlin.

\bibitem[Karatzas and Shreve, 1991]{KaratzasShreve}
Karatzas, I. and Shreve, S.~E. (1991).
\newblock {\em Brownian Motion and Stochastic Calculus}.
\newblock Springer, New York, NY, 2nd edition.

\bibitem[Kleptsina et~al., 1997]{KleptsinaLiptserSerebrovski1997}
Kleptsina, M.~L., Liptser, R.~S., and Serebrovski, A. (1997).
\newblock Nonlinear filtering problem with contamination.
\newblock {\em Annals of Applied Probability}, 7:917--934.

\bibitem[Kushner, 1990]{Kushner}
Kushner, H.~J. (1990).
\newblock {\em Weak Convergence Methods and Singularly Perturbed Stochastic
  Control and Filtering Problems}.
\newblock Birkh\"{a}user, Boston-Basel-Berlin.

\bibitem[Kutoyants, 2004]{Kutoyants}
Kutoyants, Y. (2004).
\newblock {\em Statistical Inference for Ergodic Diffusion Processes}.
\newblock Springer, London.

\bibitem[Liberzon and Brockett, 2000]{LiberzonBrockett}
Liberzon, D. and Brockett, R. (2000).
\newblock Spectral analysis of fokker-planck and related operators arising from
  linear stochastic differential equations.
\newblock {\em SIAM J. Control Optimization}, 38(5):1453--1467.

\bibitem[Linetsky, 2007]{linetskyBook}
Linetsky, V. (2007).
\newblock Spectral methods in derivative pricing.
\newblock In {\em Handbooks in Operations Research and Management Science:
  Financial Engineering}, volume~15, chapter~6, pages 223--299. Elsevier B.V.

\bibitem[Papavasiliou et~al., 2009]{PapavasileiouPaviotisStuart2009}
Papavasiliou, A., Pavliotis, G., and Stuart, A. (2009).
\newblock Maximum likelihood drift estimation for multiscale diffusions.
\newblock {\em Stochastic Processes and their Applications}, 119:3173--3210.

\bibitem[Pardoux and Veretennikov, 2003]{PardouxVeretennikov2}
Pardoux, E. and Veretennikov, A. (2003).
\newblock On {P}oisson equation and diffusion approximation ii.
\newblock {\em Annals of Probability}, 31(3):1066--1092.

\bibitem[Park et~al., 2011]{ParkRozovskySowers2010}
Park, J., Rozovsky, B., and Sowers, R. (2011).
\newblock Efficient nonlinear filtering of a singularly perturbed stochastic
  hybrid system.
\newblock {\em LMS J. Computational Mathematics}, 14:254--270.

\bibitem[Park et~al., 2008]{ParkSriSowers2008}
Park, J., Sowers, R., and Namachchivaya, N.~S. (2008).
\newblock A problem in stochastic averaging of nonlinear filters.
\newblock {\em Stochastics and Dynamics}, 8:543--560.

\bibitem[Park et~al., 2010]{ParkSriSowers2011}
Park, J., Sowers, R., and Namachchivaya, N.~S. (2010).
\newblock Dimensional reduction in nonlinear filtering.
\newblock {\em Nonlinearity}, 23:305--324.

\bibitem[Rozovskii, 1990]{Rozovskii}
Rozovskii, L. (1990).
\newblock {\em Stochastic Evolution System: Linear Theory and Aplications to
  Non-linear Filtering}.
\newblock Kluwer Academic Publishers, Dordrecht.

\bibitem[Rozovsky, 1991]{rozovsky1992}
Rozovsky, B. (1991).
\newblock A simple proof of uniqueness for {K}ushner and {Z}akai equations.
\newblock In Mayer-Wolf, E., editor, {\em Stochastic analysis}, pages 449--458.
  Boston: Academic Press.

\bibitem[Zhang, 2010]{zhang}
Zhang, F. (2010).
\newblock {High-Frequency Trading, Stock Volatility, and Price Discovery}.
\newblock {\em SSRN eLibrary}.

\end{thebibliography}

\end{document}